\documentclass[a4paper,10pt]{article}
\usepackage[latin1]{inputenc}
\usepackage[T1]{fontenc}
\usepackage{amsmath}
\usepackage{amsfonts}
\usepackage{amstext}
\usepackage{amsthm}
\usepackage[all]{xy}
\usepackage{geometry}
\usepackage{srcltx}
\usepackage{graphics}
\usepackage{epsfig}
\usepackage{subfigure}
\usepackage{slashed}
\usepackage{color}
\usepackage{bbm}
\usepackage{amssymb}
\usepackage{sidecap}
\usepackage{srcltx}

\newtheorem{theorem}{Theorem}[section]
\newtheorem{lemma}[theorem]{Lemma}
\newtheorem{defi}[theorem]{Definition}
\newtheorem{rem}[theorem]{Remark}
\newtheorem{prop}[theorem]{Proposition}
\newtheorem{cor}[theorem]{Corollary}
\newtheorem{ex}[theorem]{Example}

\DeclareMathOperator{\im}{im}

\DeclareMathOperator{\ind}{ind}

\DeclareMathOperator{\coker}{coker}

\DeclareMathOperator{\pr}{pr}

\DeclareMathOperator{\IM}{Im}
\DeclareMathOperator{\Ker}{Ker}
\DeclareMathOperator{\Coker}{Coker}

\def\id{\mathrm{id}}

\def\cl{\mathrm{cl\,}}
\def\Int{\mathrm{int\,}}
\def\bd{\mathrm{bd\,}}
\def\dist{\mathrm{dist}}

\newcommand{\N}{{\mathbb N}}
\newcommand{\R}{{\mathbb R}}
\newcommand{\C}{{\mathbb C}}
\newcommand{\Z}{{\mathbb Z}}

\newcommand{\abs}[1]{\left|#1\right|}				
\newcommand{\norm} [1]{\left\|#1\right\|}			

\title{Fredholm theory of families of discrete dynamical systems and its applications to bifurcation theory}

\author{Robert Skiba and Nils Waterstraat}

\begin{document}
\date{}
\maketitle


\begin{abstract}
In a previous work, we proved an index theorem for families of asymptotically hyperbolic discrete dynamical systems and obtained applications to bifurcation theory. A weaker and far more common assumption than asymptotic hyperbolicity is the existence of an exponential dichotomy. In this paper we generalize all our previous results to the latter setting, which requires substantial modifications of our arguments. In addition, we generalize previous results on continuity and differentiability of Nemitski operators for discrete dynamical systems to obtain even better bifurcation results.
\end{abstract}

\section{Introduction}
Let $\Lambda$ be a compact metric space and $f_n:\Lambda\times\mathbb{R}^d\rightarrow\mathbb{R}^d$, $n\in\mathbb{Z}$, a sequence of continuous maps such that $f_n(\lambda,0)=0$ for all $n\in\mathbb{Z}$ and $\lambda\in\Lambda$. We consider the discrete dynamical systems

\begin{align}\label{intro-nonlinear-equation}
\phi(n+1)&=f_n(\lambda,\phi(n)),\;n\in \Z,
\end{align}
and study bifurcation from the trivial branch $\Lambda\times\{0\}$ of homoclinic solutions, i.e., solutions of \eqref{intro-nonlinear-equation} that converge to $0$ for $n\rightarrow\pm\infty$. If each $f_n$ is differentiable in the second variable with a continuous derivative, then the linearizations of \eqref{intro-nonlinear-equation} at $0$ are the family

\begin{align}\label{intro-linear-equation}
\phi(n+1)&=\mathbb{A}_n(\lambda)\phi(n),\;n\in \Z,
\end{align}
of linear discrete dynamical systems, where $\mathbb{A}_n(\lambda)=D_2 f_n(\lambda,0)$.\\
In our previous work \cite{SkibaIch} we imposed assumptions from \cite{JacoboRobertII} on $f$ that allow to study the bifurcation problem of \eqref{intro-nonlinear-equation} by topological methods in the Banach space $\ell_0(\mathbb{R}^d)$ of all sequences converging to $0$ for $n\rightarrow\pm\infty$. A fundamental assumption in \cite{SkibaIch} is that the matrices $\mathbb{A}_n(\lambda)$ in \eqref{intro-linear-equation} are asymptotically hyperbolic, i.e., they converge for $n\rightarrow\pm\infty$ uniformly to families of hyperbolic matrices.\\
The most powerful concept for studying homoclinic trajectories of dynamical systems are so-called exponential dichotomies (ED for short), which generalize the concept of hyperbolicity from autonomous to nonautonomous linear equations (c.f., e.g.,  \cite{Pa84,Pal88,Christian10,Poetzsche, Poetzscheb}). As it can be shown that every asymptotically hyperbolic system \eqref{intro-linear-equation} has an ED, it is a natural question whether our assumptions in \cite{SkibaIch} can be weakened to this far more common setting. The aim of this article is to give an affirmative answer.\\
Naturally, discrete dynamical systems are studied in a suitable sequence space and, as we will recall below, an ED comes by definition with a bounded projection in that Banach space. As continuous families of projections induce vector bundles, this suggests that parametrized difference equations should be studied by $KO$-theory, which is a cohomology theory whose groups are built by equivalence classes of vector bundles. Consequently, the problem of existence of a continuous family of projections satisfying the conditions of an exponential dichotomy becomes important for our approach. A part of this work is concerned with this existence question and its stability under perturbations. Let us note that this problem has been studied before (see \cite{Bar, Poetzsched} and the references therein), however the obtained results were not suitable for applicable $KO$-theoretic methods. The crucial novelty in our approach is that, in contrast to \cite{Aulbach0,Aulbach1}, we only need the existence of projections on semi-axis and not on all of $\mathbb{R}$. We adapt the spectral theory from \cite{Aulbach0,Aulbach1} in order to obtain a direct construction of continuous families of projections on both semi-axes for perturbed parametrized difference equations.\\
The central concept that we need to link the families of projections with the difference equations \eqref{intro-nonlinear-equation} and \eqref{intro-linear-equation} is the index bundle of Atiyah and J\"anich, which is a generalization of the index of Fredholm operators to families. The index bundle is a $KO$-theory class of the underlying parameter space, which formally shares all properties of the integer-valued index of a single operator. The main theorem of this article is a formula that computes the index bundle for discrete dynamical systems \eqref{intro-linear-equation} having an ED in terms of the associated projections. In the special case that \eqref{intro-linear-equation} is asymptotically hyperbolic, we reobtain the family index theorem that we proved in \cite{SkibaIch}, where we need the perturbation theory for EDs that we previously established.\\
In the final part of this work, we apply the obtained index theorem to bifurcation theory along the lines of our previous work \cite{SkibaIch}, which substantially widens its applicability. Indeed, the only example that we could give in \cite{SkibaIch} was a two dimensional system parametrized by a torus that was adapted from an example of \cite{JacoboAMS}. Here we use our new approach to construct a whole class of examples for general parameter spaces by perturbing simple systems. Finally, as we are dealing with nonlinear systems in the bifurcation setting, we in particular need to consider differentiability of Nemitski operators for (generally) non-invertible difference equations, where we also generalize previous results from \cite{SkibaIch} to make our theory more applicable.\\
The paper is organized as follows. We first fix some notations and recall some preliminaries in the following second section. Section 3 is devoted to continuity and differentiability of the Nemitski operator in our setting. In Section 4 we recall the classical concept of exponential dichotomy and compare it to asymptotically hyperbolic systems which we considered in our previous work \cite{SkibaIch}. Moreover, we shall provide a direct construction of the continuous family of projections for given parametrized discrete dynamical systems. The mentioned family of projections plays a crucial role in the applications of $KO$-theory to discrete dynamical systems. In Section 5 we study Fredholm properties of the linearized equations and prove the index theorem for \eqref{intro-linear-equation}, which is the main result of this work. In the final Section 6, we use the results of the Sections 3--5 to obtain generalizations of the bifurcation results from \cite{SkibaIch} including a whole class of new and non-trivial examples.

\section{Notation and Preliminaries}
The aim of this section is to recall some basic concepts and notations
that we use throughout the paper. We denote by $\mathcal{L}(X)$ the space of bounded linear operators on a Banach space $X$ with the operator norm, and by $GL(X)$ the open subset of all invertible elements. The symbol $I_X$ stands for the identity operator on $X$. If $X=\mathbb{K}^d$ for $\mathbb{K}=\mathbb{C}$ or $\mathbb{K}=\mathbb{R}$ and some $d\in\mathbb{N}$, then we use instead the common notation $M(d,\mathbb{K})$, $GL(d,\mathbb{K})$ and $I_d$, respectively. The norm in the Euclidean space $\mathbb{K}^d$ will also be denoted by the symbol $|\cdot|$. As usual, $\Z$ denotes the ring of integers, $\N$ are the positive integers including zero. Given $\kappa\in \Z$, we define the discrete intervals $\Z^+_{\kappa}:=\Z\cap [\kappa,+\infty)$, $\Z^-_{\kappa}:=\Z\cap (-\infty,\kappa]$. \newline\indent
A matrix  $\mathcal{A}\in M(d,\mathbb{C})$ is called {\it hyperbolic} if $\mathcal{A}$ has no
eigenvalues of modulus one, i.e., $\sigma(A)\cap
\{|z|=1\}=\emptyset.$ Consequently, the spectrum $\sigma(\mathcal{A})$ of a hyperbolic
matrix $\mathcal{A}$ consists of the two disjoint closed subsets
$\sigma(\mathcal{A})\cap\{|z| < 1\}$ and $\sigma(\mathcal{A})\cap\{|z|> 1\}$. Note that we do not assume that a hyperbolic matrix is invertible. We denote the set of all hyperbolic matrices by $H(d,\mathbb{C})$.\\
In what follows, the symbol $\mathbb{I}$ stands for one of the sets $\Z^{\pm}_{\kappa}$ for some $\kappa\in \Z$.  A map $\mathbb{P}\colon \mathbb{I}\to \mathcal{L}(X)$ is called a projector provided for any $n\in \mathbb{I}$, $\mathbb{P}(n)\colon X\to X$ satisfies the condition $\mathbb{P}(n)\circ \mathbb{P}(n)=\mathbb{P}(n)$. The symbol $\mathbb{P}(n)$ will sometimes be denoted, for simplicity, by $\mathbb{P}_n$. Let $\Lambda$ be a metric space. By a parametrized projector we shall mean a map $\mathbb{P}\colon \Lambda\times \mathbb{I}\to \mathcal{L}(X)$ such that $\mathbb{P}(\lambda)\colon \mathbb{I}\to \mathcal{L}(X)$ is a projector for all $\lambda\in\Lambda$, and $\mathbb{P}_n\colon \Lambda\to \mathcal{L}(X)$ is continuous for all $n\in \mathbb{I}$.\\
We write $\ell^\infty(\mathbb{K}^d)$ for the Banach space of bounded sequences $\phi=(\phi(n))_{n\in\Z}$ in $\mathbb{K}^d$ (or equivalently, bounded functions $\phi\colon \Z\to \mathbb{K}^d$) with norm $\norm{\phi}_{\infty}:=\sup_{n\in\Z}\abs{\phi(n)}$, where $\mathbb{K}=\mathbb{C}$ or $\mathbb{K}=\R$. Sometimes we will write $\phi_n$ instead of $\phi(n)$. The set $\ell_0(\mathbb{K}^d)$ of all sequences with two-sided limit $0$ is a closed subspace of $\ell^\infty(\mathbb{K}^d)$. In what follows, we shall often make use of the following closed subspaces of $\ell_0(\mathbb{K}^d)$:
\begin{align*}
\ell_{0\kappa}^{\pm}(\mathbb{K}^d)&:=\{\mathbbm{1}_{\Z^{\pm}_{\kappa}}\phi \mid \phi\in \ell_0(\mathbb{K}^d)\},\quad
\ell_{0}^{\pm}(\mathbb{K}^d):=\{\mathbbm{1}_{\Z^{\pm}_{0}}\phi\mid \phi\in \ell_0(\mathbb{K}^d)\},\\
\ell^{\infty}_{\pm}(\mathbb{K}^d)&:=\{\mathbbm{1}_{\Z^{\pm}_{0}}\phi\mid\phi\in \ell^{\infty}(\mathbb{K}^d)\},\quad
\ell_{\underline{\kappa}\overline{\kappa}}(\mathbb{K}^d):=\ell^+_{0\underline{\kappa}}(\mathbb{K}^d)\cap
\ell_{0\overline{\kappa}}^{-}(\mathbb{K}^d),
\end{align*}
where $\underline{\kappa}\leqslant 0\leqslant\overline{\kappa}$ and $\mathbbm{1}_{\Z^{\pm}_{m}}\phi$ is defined by
$(\mathbbm{1}_{\Z^{\pm}_{m}}\phi)(n):=\mathbbm{1}_{\Z^{\pm}_{m}}(n)\cdot\phi(n)$,
for all $n\in\Z$, $m\in\Z$.\newline\indent
For any linear bounded map $T\colon \mathbb{E}\to \mathbb{E}$, where $\mathbb{E}$ is a normed space and $v\in \mathbb{E}$, we define
$
|||T^{-k}v|||:=\inf\{\|w\|\mid T^k(w)=v\}$,
where $k>0$. If $T^{-k}(v)=\emptyset$, then we put $|||T^{-k}v|||:=\infty$ (and hence $\sqrt[k]{|||T^{-k}v|||}=\infty$).
We shall justify that if $v\neq 0$ and $T^{-k}(v)\neq\emptyset$, then $|||T^{-k}v|||>0$. To see this, let us first observe that $|||T^{-k}v|||\geqslant 0$ for all $v\in \mathbb{E}$ with $v\neq 0$. Now assume on the contrary that $|||T^{-k}v|||=0$ for some $v\neq 0$. Then there exists a sequence $(w_n)\in \mathbb{E}$ such that $T^k(w_n)=v$ and $\|w_n\|\rightarrow 0$ as $n\to\infty$. Consequently,
$0<\|v\|=\|T^k(w_n)\|\leqslant \|T^k\|\|w_n\|\rightarrow 0$ as $n\to\infty$
and hence we get a contradiction, which proves that $|||T^{-k}v|||> 0$.\newline\indent
Given a metric space $(\mathbf{X},\mathrm{d}_{\mathbf{X}})$,
$\mathbf{A}\subset \mathbf{X}$, we will denote the closure, interior and the
boundary of $\mathbf{A}$ in $\mathbf{X}$ by $\cl \mathbf{A}$, $\Int \mathbf{A}$ and $\bd \mathbf{A}$,
respectively. Furthermore, by
$D_{\mathbf{X}}(x,r)$ and $B_{\mathbf{X}}(x,r)$ we denote the closed and open ball
around $x$ of radius $r$ in $\mathbf{X}$, respectively, and, for $\varepsilon>0$,
$O_{\varepsilon}(\mathbf{A}):=\{x\in \mathbf{X}\mid \dist(x,\mathbf{A}):=\inf_{a\in \mathbf{A}}\mathrm{d}_{\mathbf{X}}(x,a)<\varepsilon\}$ is the $\varepsilon$-neighborhood of $\mathbf{A}$.
For $\mathbf{X}=\R^d$, an open (resp. closed) ball of radius $r>0$ centred in
$x\in \R^d$ is denoted by $B_d(x,r)$ (resp. $D_d(x,r)$).\newline\indent
Now we introduce some notions from K-theory. Let $\Lambda$ stand for a path-connected and compact metric space.
Recall that the trivial vector bundle over $\Lambda$ with fiber $V$ is the product $\Lambda\times V$ with the projection onto the first component. It will be denoted by the symbol $\Theta(V)$. Define two vector bundles $\mathbb{E}_1$ and $\mathbb{E}_2$ over $\Lambda$ to be stably isomorphic, written $\mathbb{E}_1\approx_s \mathbb{E}_2$,  if $\mathbb{E}_1\oplus \Theta(\R^n)$ is isomorphic to $\mathbb{E}_2\oplus \Theta(\R^n)$  for some non-negative integer $n$. It is an easy exercise to check
that $\approx_s$ is an equivalence relation. The set of stable equivalence classes of vector bundles of finite rank over $\Lambda$ forms a commutative semigroup with respect to direct sum $\oplus$. Now by applying the Grothendieck construction to this semigroup we obtain an abelian group $KO(\Lambda)$
consisting of formal differences $[\mathbb{E}]-[\mathbb{F}]$ of isomorphism classes $[\mathbb{E}]$, $[\mathbb{F}]$ of vector bundles $\mathbb{E}$ and $\mathbb{F}$ over $\Lambda$, with the equivalence relation
$[\mathbb{E}]-[\mathbb{F}]=[\mathbb{E}']-[\mathbb{F}']$ if and only if $\mathbb{E}\oplus \mathbb{F}'\approx_s \mathbb{E}'\oplus \mathbb{F}$. Furthermore, the addition rule on $KO(\Lambda)$ is defined by
$([\mathbb{E}]-[\mathbb{F}])+([\mathbb{E}']-[\mathbb{F}'])=[\mathbb{E}\oplus \mathbb{E}']-[\mathbb{F}\oplus \mathbb{F}'].$
Note that the zero element in $KO(\Lambda)$ is of the form $[\mathbb{E}]-[\mathbb{E}]$, for any class $[\mathbb{E}]$, and the inverse element of $[\mathbb{E}]-[\mathbb{F}]$ is $[\mathbb{F}]-[\mathbb{E}]$.
The elements of $KO(\Lambda)$ are called virtual bundles. One can show that each virtual bundle can be represented as a difference $[\mathbb{E}]-[\Theta(\R^n)]$ for some $[\mathbb{E}]$ and $n\in\N$. Moreover, KO-theory can be regarded as a contravariant functor from the category of compact topological spaces to the category of abelian groups.
This follows from the fact that every continuous function $f\colon \Lambda\to \Lambda'$ induces a group homomorphism
$f^{\ast}\colon KO(\Lambda')\to KO(\Lambda)$ sending $[\mathbb{E}]-[\mathbb{E}']$ to $[f^{\ast}(\mathbb{E})]-[f^{\ast}(\mathbb{E}')]$, where $f^{\ast}(\mathbb{E})$ and $f^{\ast}(\mathbb{E}')$ are the pullback bundles. The reduced $KO$-group of $\Lambda$, denoted by $\widetilde{KO}(\Lambda)$, is defined to be the kernel of the homomorphism
$KO(\Lambda)\to KO(\ast)$ induced by the inclusion $\ast\hookrightarrow \Lambda$, where $\ast$ denotes any point of $\Lambda$. Notice that this definition does not depend on the choice of the point $\ast$ of $\Lambda$ since $\Lambda$ is connected. Furthermore, it is easy to see that $\widetilde{KO}(\Lambda)$
consists of all $[\mathbb{E}]-[\mathbb{F}]\in KO(\Lambda)$ such that $\mathbb{E}$ and $\mathbb{F}$ both have the same dimension, and we have a splitting $KO(\Lambda)=\widetilde{KO}(\Lambda)\oplus \Z$.\newline\indent
Two vectors bundles $\mathbb{E}$ and $\mathbb{F}$ over $\Lambda$ are called fiberwise homotopy equivalent if
there is a fibre preserving homotopy equivalence between their sphere bundles $S(\mathbb{E},1)$
and $S(\mathbb{F},1)$. In addition, $\mathbb{E}$ and $\mathbb{F}$ are stably fibrewise homotopy equivalent if $\mathbb{E}\oplus\Theta(\R^n)$ and
$\mathbb{F}\oplus\Theta(\R^m)$ are fibrewise homotopy equivalent for some non-negative integers $n,m$. Letting $T(\Lambda)$ be the subgroup of
$\widetilde{KO}(\Lambda)$ generated by elements $[\mathbb{E}]-[\mathbb{F}]$ such that $\mathbb{E}$ and $\mathbb{F}$ are stably fiberwise
homotopy equivalent, we define the group $J(\Lambda)$ as the quotient of $\widetilde{KO}(\Lambda)$ by the subgroup $T(\Lambda)$. The generalized $J$-homomorphism  $J\colon \widetilde{KO}(\Lambda)\to J(\Lambda)$ is the projection to the quotient (cf. \cite{KTheoryAtiyah}). Recall that the group $J(\Lambda)$ was introduced by Atiyah in \cite{ThomAtiyah}, where it is also proved that $J(\Lambda)$ is a finite
group provided $\Lambda$ is a finite CW-complex. \newline\indent We conclude our recap of $K$-theory by the following simple example: $\widetilde{KO}(S^1)$ is a cyclic group  of order two and the generator for this group is the M\"{o}bius bundle. This is because the direct sum of two copies of the M\"{o}bius bundle is the trivial bundle $S^1\times\R^2$, and hence it is readily seen that $T(S^1)$ is the trivial subgroup of $\widetilde{KO}(S^1)$. Consequently, $J\colon \widetilde{KO}(S^1)\to J(S^1)$ is an isomorphism.\newline\indent
 Finally, let us explain more precisely what we mean by bifurcation from the trivial branch. Let $X,Y$ be normed spaces and let $\Lambda$ be a path-connected and compact metric space. Let $f\colon \Lambda\times \mathcal{O}\to Y$ be a continuous map, where $\mathcal{O}$ is a subset of $X$ containing $0\in X$. We will assume that $f(\lambda,0)=0$ for all $\lambda\in\Lambda$. Solution of the equation $f(\lambda,x)=0$ of the form $(\lambda,0)$ will be called trivial and the set $\Lambda\times\{0\}$ will be called the trivial branch. If the equation $f(\lambda,x)=0$ has a nontrivial solution $(\lambda,x)\in \Lambda\times \mathcal{O}$, $x\neq 0$, in every neighborhood of $(\lambda_{\ast},0)$, then $(\lambda_{\ast},0)$ is said to be a bifurcation point for $f(\lambda,x)=0$ (in what follows $(\lambda_{\ast},0)$ will be frequently identified with the point $\lambda_{\ast}$). The set of bifurcation points will be denoted by the symbol $\mathcal{B}_f$. 

\section{Nemitski operators for noninvertible difference equations}\label{section-Nemitski}
In this section we are going to present continuity and differentiability properties of functional operators associated with difference equations, which are usually called Nemitski operators. Let us first introduce the type of difference equations that will be in the center of our interest. Namely, we focus on the parametrized difference equations
\begin{align}
\phi(n+1)&=f_n(\lambda,\phi(n)),\;n\in \Z,\tag{F}\label{nonlinear-equation}
\end{align}
where $\lambda\in\Lambda$ is a parameter, the right-hand side $f_n\colon \Lambda\times \R^d\to \R^d$ is continuous for all $n\in\Z$ and $f_n(\lambda,0)=0$ for all $n\in\Z$ and $\lambda\in\Lambda$. A doubly infinite sequence of maps $f_n$ will be denoted by the symbol $f\colon \Lambda\times\Z\times\R^d\to \R^d$. \newline\indent
For a fixed parameter $\lambda\in\Lambda$, a solution of the difference equation \eqref{nonlinear-equation}, or a trajectory of $f$, is a sequence $\phi=(\phi(n))_{n\in\Z}$ with $\phi(n)\in \R^d$ satisfying the recursion \eqref{nonlinear-equation}. A trajectory $(\phi(n))_{n\in\Z}$ of \eqref{nonlinear-equation} will be called {\it homoclinic} to $0$, or simply a homoclinic trajectory, if $\phi(n)\to 0$ as $|n|\to \infty$, while a constant trajectory of \eqref{nonlinear-equation} will be called {\it stationary}. Since $f_n(\lambda,0)=0$ for all $n\in\Z$, the constant sequence $\mathbf{0}=(0)_{n\in\Z}$ is a stationary trajectory of $f$ and homoclinic to $0$. In what follows, we will be interested in nontrivial trajectories homoclinic to $0$.\\
In order to study the parametrized difference equations \eqref{nonlinear-equation} in the context of Fredholm and bifurcation theory, we will additionally suppose the following assumptions:
\begin{enumerate}
\item[$(F0)$] $f_n$ is differentiable in the second variable with the derivative $D_2 f_n$ depending continuously on $\Lambda\times\R^d$,
\item[$(F1)$] there exists $r_0>0$ such that the family $\{f_n\colon \Lambda\times D_d(0,r_0)\to \R^d\}_{n\in \Z}$ is equicontinuous and the family $\{D_2f_n\colon \Lambda\times D_d(0,r_0)\to \mathcal{L}(\R^d)\}_{n\in \Z}$ is equicontinuous and pointwise bounded.
\end{enumerate}
Before proceeding, let us make some comments about these assumptions.
\begin{itemize}
\item $(F1)$ is weaker than the corresponding assumption in \cite{SkibaIch}. Indeed, in the paper \cite{SkibaIch}
it was assumed that the families $\{f_n\}_{n\in\Z}$ and $\{D_2f_n\}_{n\in \Z}$ are equicontinuous on all subsets of the form $\Lambda\times K$, where $K$ is a compact subset of $\R^d$. Here we merely assume that the families $\{f_n\}_{n\in\Z}$ and $\{D_2f_n\}_{n\in \Z}$ are equicontinuous on a set of the form $\Lambda\times D_d(0,r_0)$ for an arbitrarily small closed ball $D_d(0,r_0)$ in $\mathbb{R}^d$.
\item Since $\Lambda\times D_d(0,r_0)$ is compact, $(F1)$ implies that $\{f_n\}_{n\in\Z}$ and $\{D_2f_n\}_{n\in\Z}$ are uniformly equicontinuous on $\Lambda\times D_d(0,r_0)$, and $\{D_2f_n\}_{n\in\Z}$ is uniformly bounded on $\Lambda\times D_d(0,r_0)$.
\item If $\Lambda$ is a smooth compact manifold and $f_n\colon \Lambda\times\R^d\to \R^d$ is differentiable on $\Lambda\times\R^d$ for all $n\in\Z$, then Assumption $(F1)$ can be formulated as follows: there exists $r_0>0$ such that the family $\{Df_n\colon \Lambda\times D_d(0,r_0)\to \mathcal{L}(\R^d)\}_{n\in \Z}$ is equicontinuous and pointwise bounded, where the symbol $Df_n$ denotes the derivative of $f_n$.
\end{itemize}
In what follows, we also consider linear difference equations
\begin{align}\label{linear-equation}\tag{$L$}
\phi(n+1)&=\mathbb{A}_n(\lambda)\phi(n),\;n\in \Z,
\end{align}
where the right-hand side $\mathbb{A}\colon \Lambda\times\Z\to \mathcal{L}(\R^d)$ satisfies the assumption
\begin{enumerate}
\item[$(L0)$] the family $\mathbb{A}=\{\mathbb{A}_n\colon \Lambda\to \mathcal{L}(\R^d)\}_{n\in\Z}$ is equicontinuous and pointwise bounded.
\end{enumerate}
Note that the compactness of the parameter space $\Lambda$ implies that the family $\mathbb{A}$ from $(L0)$ is uniformly
equicontinuous and uniformly bounded.\newline\noindent
The maps $\mathbb{A}\colon \Lambda\times\Z\to \mathcal{L}(\R^d)$ and $f\colon \Lambda\times\Z\times\R^d\to \R^d$ will be called {\it a parametrized linear discrete vector field} and {\it a parametrized nonlinear discrete vector field}, respectively.\\
We will see below that the assumptions (F0), (F1) and (L0) in particular imply that the difference equations \eqref{linear-equation} and \eqref{nonlinear-equation} induce  functional operators
\begin{align}
&\mathcal{N}_{\mathbb{A}}\colon \Lambda\to \mathcal{L}(\ell^{\infty}(\R^d)), & \mathcal{N}_{\mathbb{A}}(\lambda)\phi(n)&:=\mathbb{A}_n(\lambda)\phi(n),\\
&\mathcal{N}_{f}\colon \Lambda\times \ell_0(\R^d)\to \ell_0(\R^d), &
\mathcal{N}_{f}(\lambda,\phi)(n)&:=f_n(\lambda,\phi(n)),
\end{align}
that are called {\it Nemitski operators}. In what follows, we will also consider the restriction

\[\mathcal{N}_{f(\lambda)}\colon \ell_0(\R^d)\to \ell_0(\R^d),\quad\mathcal{N}_{f(\lambda)}\phi=\mathcal{N}_{f}(\lambda,\phi).\]
Note that the homoclinic solutions of \eqref{nonlinear-equation} are strictly related to the nonlinear functional operator
$\mathbb{S}_l-\mathcal{N}_{f}\colon \Lambda\times\ell_0(\R^d)\to \ell_0(\R^d)$,
where $\mathbb{S}_l\colon \Lambda\times\ell_0(\R^d)\to \ell_0(\R^d)$ is the left shift operator given by $(\mathbb{S}_l(\lambda,\phi))(n)=\phi(n+1)$ for all
$\lambda\in\Lambda$ and $\phi\in \ell_0(\R^d)$. Indeed, given a parameter $\lambda$, $\phi\in \ell_0(\R^d)$ is a solution to \eqref{nonlinear-equation} if and only if
\begin{align*}
(\mathbb{S}_l-N_{f})(\phi)=\mathbb{S}_l\phi-\mathbb{N}_f(\lambda,\phi)=0,
\end{align*}
what explains why it is worth to study the Nemitski operators. Consequently, in the remaining part of this section we consider continuity and differentiability properties of $\mathcal{N}_{\mathbb{A}}$ and $\mathcal{N}_{f}$. We begin with the following simple result that we leave to the reader.

\begin{lemma}\label{N-L-continuous}
If $\mathbb{A}\colon \Lambda\times\Z\to \mathcal{L}(\R^d)$ satisfies $(L0)$, then the Nemitski operator\linebreak $\mathcal{N}_{\mathbb{A}}\colon \Lambda\to \mathcal{L}(\ell^{\infty}(\R^d))$ is well-defined, continuous on $\Lambda$ and $\mathcal{N}_{\mathbb{A}}(\lambda)(\ell_0(\R^d))\subset \ell_0(\R^d)$ for all $\lambda\in\Lambda$.
\end{lemma}
\noindent
The discussion of the continuity and differentiability properties of the Nemitski operator $\mathcal{N}_f$ is more involved as $f_{n}(\lambda,\cdot)\colon \R^d\to \R^d$ is in general nonlinear, in contrast to $\mathbb{A}_{n}(\lambda)\in \mathcal{L}(\R^d)$.

\begin{lemma}\label{L-Niemytski}
If $f\colon \Lambda\times\Z\times \R^d\to \R^d$ satisfies $(F0)$--$(F1)$, then the Nemitski operator

\[\mathcal{N}_{f}\colon \Lambda\times \ell_0(\R^d)\to \ell_0(\R^d)\]
is
\begin{enumerate}
\item[$(a)$] well-defined and continuous on $\Lambda\times \ell_0(\R^d)$,
\item[$(b)$] differentiable with respect to the second variable, and the corresponding map of derivatives
$(\lambda,\phi)\mapsto D_2\mathcal{N}_{f}(\lambda,\phi)$ is continuous.
\end{enumerate}
\end{lemma}
\begin{proof}
To prove $(a)$, we take a point $\phi\in \ell_0(\R^d)$ and a parameter $\lambda\in\Lambda$. Then there exists $m_0\in \N$ such that $\phi(n)\in B_d(0,r_0)$
for all $|n|\geqslant m_0$. Consequently, the Mean Value Theorem and $(F1)$ imply that for $|n|\geqslant m_0$
\begin{align}
\begin{split}
|(\mathcal{N}_f(\lambda,\phi))(n)|&=|f_n(\lambda,\phi(n))|=|f_n(\lambda,\phi(n))-f_n(\lambda,0)|\\
&\leqslant \Big(\sup_{|n|\geqslant m_0,s\in [0,1]}\|(D_2f_n)(\lambda,s\phi(n))\|\Big)\,|\phi(n)|\\
&\leqslant \Big(\sup_{|n|\geqslant m_0,x\in D_d(0,r_0)}\|(D_2f_n)(\lambda,x)\|\Big)\,|\phi(n)|,
\end{split}
\end{align}
which converges to $0$ as $n\rightarrow\pm\infty$. Thus $\mathcal{N}_f(\lambda,\phi)\in\ell_0(\R^d)$ showing that $\mathcal{N}_f$ is well defined.\\
Next we prove that $\mathcal{N}_f$ is continuous at $(\lambda_0,\phi_0)$.
To see this, note first that $\{f_n\colon \Lambda\times D_d(0,r_0)\to \R^d\}$ is uniformly equicontinuous, which means that for every $\varepsilon>0$ there exists $\delta<r_0/2$ such that for all $\lambda_1,\lambda_2\in\Lambda$ and $x,y\in D_d(0,r_0)$ one has
\begin{align}\label{delta}
\mathrm{d}_{\Lambda}(\lambda_1,\lambda_2)<\delta\wedge |x-y|<\delta\Longrightarrow |f_n(\lambda_1,x)-f_n(\lambda_2,y)|<\varepsilon/2, \text{ for all }n\in\Z.
\end{align}
Fix $\varepsilon/2$. Let $\delta_1<r_0/2$ be as in \eqref{delta}. Since $\phi\in \ell_0(\R^d)$, there exists $m_0\in \N$ such that $\phi(n)\in B_d(0,r_0/2)$ for all $|n|> m_0$. Thus if $\|\psi-\phi_0\|_{\infty}<\delta_1$ with $\psi\in \ell_0(\R^d)$, then $\psi(n)\in B_d(0,r_0)$ for all $|n|> m_0$. Combining this with \eqref{delta} we see that
\begin{align*}
\mathrm{d}_{\Lambda}(\lambda_0,\lambda_1)<\delta_1\wedge \|\phi_0-\psi\|_{\infty}<\delta_1\Longrightarrow |f_n(\lambda_0,\phi_0(n))-f_n(\lambda_1,\psi(n))|<\varepsilon/2, \text{ for }|n|> m_0.
\end{align*}
By Assumption $(F0)$, the function $f_i\colon \Lambda\times\R^d\to\R^d$ is continuous at the point $(\lambda_0,\phi_0(i))$ for any $|i|\leqslant m_0$. Hence, since this family is finite, there is $\delta_2>0$ such that
\begin{align*}
\mathrm{d}_{\Lambda}(\lambda_0,\lambda_1)<\delta_2\wedge \|\phi_0-\psi\|_{\infty}<\delta_2\Longrightarrow |f_i(\lambda_0,\phi_0(i))-f_i(\lambda_1,\psi(i))|<\varepsilon/2, \text{ for }|i|\leqslant m_0.
\end{align*}
Finally, letting $\delta$ be the minimum of $\delta_1$ and $\delta_2$, we obtain
\begin{align*}
\mathrm{d}_{\Lambda}(\lambda_0,\lambda_1)<\delta\wedge \|\phi_0-\psi\|_{\infty}<\delta\Longrightarrow \sup_{n\in\Z}|f_n(\lambda_0,\phi_0(n))-f_n(\lambda_1,\psi(n))|\leqslant\varepsilon/2<\varepsilon,
\end{align*}
which shows that $\mathcal{N}_f$ is continuous at $(\lambda_0,\phi_0)$ because of
\begin{align*}
\|\mathcal{N}_f(\lambda_0,\phi_0)-\mathcal{N}_f(\lambda_1,\psi)\|_{\infty}=\sup_{n\in\Z}|f_n(\lambda_0,\phi_0(n))-f_n(\lambda_1,\psi(n))|.
\end{align*}
For $(b)$, we introduce a map $\mathbb{A}\colon\Lambda\times \ell_0(\R^d)\rightarrow\mathcal{L}(\ell_0(\R^d))$ by
\begin{equation*}
(\mathbb{A}(\lambda,\phi)\psi)(n):=(D_2f_n)(\lambda,\phi(n))\psi(n),\quad n\in\mathbb{Z},
\end{equation*}
which is readily seen to be well-defined. Indeed, fix $\lambda_0\in\Lambda$ and $\phi_0\in\ell_0(\R^d)$. Since $\phi_0\in\ell_0(\R^d)$, there exists
$m_0\in \N$ such that $\phi_0(n)\in D_d(0,r_0)$ for $|n|>m_0$. Thus Assumption $(F1)$ induces that for $|n|>m_0$ one has
\begin{align*}
|(\mathbb{A}(\lambda_0,\phi_0)\psi)(n)|&=|(D_2f_n)(\lambda_0,\phi_0(n))\psi(n)|\leqslant \|(D_2f_n)(\lambda_0,\phi_0(n))\|\cdot|\psi(n)|\\
&\leqslant \Big(\sup_{x\in D_d(0,r_0)}\|(D_2f_n)(\lambda_0,x)\|\Big)|\psi(n)|,
\end{align*}
which converges to $0$ as $n\rightarrow\pm\infty$.
Moreover, it follows as in part $(a)$ that $\mathbb{A}$ is continuous. Our aim is now to show that $D_2\mathcal{N}_{f}(\lambda,\phi)=\mathbb{A}(\lambda,\phi)$ for every fixed $\lambda\in\Lambda$ and $\phi\in \ell_0(\R^d)$. As $\mathbb{A}$ is continuous, this shows that $\mathcal{N}_f$ is differentiable in the second variable and $D_2\mathcal{N}(\lambda,\phi)$ depends continuously on $(\lambda,\phi)$.
Fix $\phi\in\ell_0(\R^d)$ and $\lambda\in \Lambda$. As above, there exists $m_0\in \N$ such that $\phi(n)\in B_d(0,r_0/2)$ for all $|n|> m_0$. Now take any $\mathfrak{h}\in \ell_0(\R^d)$. Then we have
\begin{align*}
r_\lambda(\phi,\mathfrak{h})&:=\|\mathcal{N}_{f}(\lambda,\phi+\mathfrak{h})-\mathcal{N}_{f}(\lambda,\phi)-\mathbb{A}(\lambda,\phi)\mathfrak{h}\|_{\infty}\\&
=\sup_{n\in\mathbb{Z}}|f_n(\lambda,\phi(n)+\mathfrak{h}(n))-f_n(\lambda,\phi(n))-(D_2f_n)(\lambda,\phi(n))\mathfrak{h}(n)|\\
&=\sup_{n\in\Z}\left |\int^1_0{(D_2f_n)(\lambda,\phi(n)+s\cdot\mathfrak{h}(n))\mathfrak{h}(n)\,ds-(D_2f_n)(\lambda,\phi(n))\mathfrak{h}(n)}\right|\\
&\leqslant\left[\sup_{n\in\mathbb{Z}}\left(\int^1_0{\left\|(D_2f_n)(\lambda,\phi(n)+s\cdot \mathfrak{h}(n))-(D_2f_n)(\lambda,\phi(n))\right\|ds}\right)\right]
\left(\sup_{n\in\mathbb{Z}}|\mathfrak{h}(n)|\right)\\
&\leqslant \|\mathfrak{h}\|_{\infty}\int^1_0\sup_{n\in\mathbb{Z}}\left\|(D_2f_n)(\lambda,\phi(n)+s\cdot\mathfrak{h}(n))-(D_2f_n)(\lambda,\phi(n))\right\|ds\\
&\leqslant\|\mathfrak{h}\|_{\infty}\sup_{n\in\Z}\sup_{0\leqslant s\leqslant 1}\left\|(D_2f_n)(\lambda,\phi(n)+s\cdot \mathfrak{h}(n))-(D_2f_n)(\lambda,\phi(n))\right\|.
\end{align*}
Hence
\begin{align}\label{Dif0}
\begin{split}
0\leqslant\frac{r_\lambda(\phi,\mathfrak{h})}{\|\mathfrak{h}\|_{\infty}}&\leqslant \sup_{n\in\mathbb{Z}}\sup_{0\leqslant s\leqslant 1}\left\|(D_2f_n)(\lambda,\phi(n)+s\cdot \mathfrak{h}(n))-(D_2f_n)(\lambda,\phi(n))\right\|\\
&\leqslant \sup_{|n|>m_0}\sup_{0\leqslant s\leqslant 1}\left\|(D_2f_n)(\lambda,\phi(n)+s\cdot \mathfrak{h}(n))-(D_2f_n)(\lambda,\phi(n))\right\|\\
&+\sum_{i=-m_0}^{m_0}\sup_{0\leqslant s\leqslant 1}\left\|(D_2f_n)(\lambda,\phi(n)+s\cdot \mathfrak{h}(n))-(D_2f_n)(\lambda,\phi(n))\right\|.
\end{split}
\end{align}
Note that if $\|\mathfrak{h}\|<r_0/2$, then $\phi(n)+\mathfrak{h}(n)\in D_d(0,r_0)$ for all $|n|>m_0$. Since, by Assumption $(F1)$, the family
$\{D_2f_n\colon \Lambda\times D_d(0,r_0)\to \mathcal{L}(\R^d)\}_{|n|>m_0}$ is uniformly equicontinuous, it follows that
\begin{align}\label{Dif1}
\sup_{|n|>m_0}\sup_{0\leqslant s\leqslant 1}\left\|(D_2f_n)(\lambda,\phi(n)+s\cdot \mathfrak{h}(n))-(D_2f_n)(\lambda,\phi(n))\right\|
\xrightarrow[\mathfrak{h}\rightarrow 0]{}0.
\end{align}
On the other hand, as the finite family $\{D_2f_i\colon \Lambda\times\R^d\to \mathcal{L}(\R^d)\}_{|i|\leqslant m_0}$ of continuous functions is uniformly equicontinuous on any
subset $\Lambda\times K$, where $K$ is a compact subset of $\R^d$, and hence in particular on $\Lambda\times D_d(0,\|\phi\|_{\infty}+r_0)$, it follows that
\begin{align}\label{Dif2}
\sum_{i=-m_0}^{m_0}\sup_{0\leqslant s\leqslant 1}\left\|(D_2f_n)(\lambda,\phi(n)+s\cdot \mathfrak{h}(n))-(D_2f_n)(\lambda,\phi(n))\right\|\xrightarrow[\mathfrak{h}\rightarrow 0]{}0.
\end{align}
Finally, taking into account \eqref{Dif0},\eqref{Dif1} and \eqref{Dif2}, we see that
\begin{align*}
\lim_{\|\mathfrak{h}\|\to 0}\frac{r_\lambda(\phi,\mathfrak{h})}{\|\mathfrak{h}\|_{\infty}}=0,
\end{align*}
which completes the proof of the differentiability of $\mathcal{N}_{f}$.
\end{proof}
\noindent
Assumption $(F0)$ allows us to consider the linearization of \eqref{nonlinear-equation} at $0$, i.e., the linear discrete vector field $\mathbb{A}\colon \Lambda\times\Z\to \mathcal{L}(\R^d)$ given by
\begin{equation}\label{assumption-A}
\mathbb{A}_n(\lambda):=D_2f_n(\lambda,0), \text{ for all }n\in\Z \text{ and }\lambda\in\Lambda.
\end{equation}
Further to $(F0)$--$(F1)$ above, we will introduce in Section 4 the following assumptions about the family $\mathbb{A}$ in \eqref{assumption-A} that will be required in our main theorems:
\begin{enumerate}
\item[$(F2)$] The linear vector field $\mathbb{A}\colon \Lambda\times\Z\to \mathcal{L}(\R^d)$ has an ED both on $\Z^{+}_{\overline{\kappa}}$ and
on $\Z^{-}_{\underline{\kappa}}$ for some $\underline{\kappa}<0<\overline{\kappa}$.
\item[$(F3)$] There exists $\lambda_0\in\Lambda$ such that $\mathbb{A}(\lambda_0)$ has an ED on all of $\mathbb{Z}$.
\end{enumerate}
Let us emphasize that we assumed in our previous work \cite{SkibaIch} that the parametrized linear vector field $\mathbb{A}\colon \Lambda\times\Z\to \mathcal{L}(\R^d)$ is asymptotically hyperbolic (with invertible limits), which means:
\begin{itemize}
\item As $n\to \pm\infty$, the sequence $(\mathbb{A}_n(\lambda))_{n\in\Z}$ converges uniformly with respect to $\lambda\in\Lambda$ to a family of matrices $\mathcal{A}(\lambda,\pm\infty)\in H(d,\R)\cap GL(d,\R)$.
\end{itemize}
Note that $(L0)$ together with the asymptotic hyperbolicity in particular implies that the functions $\Lambda\ni\lambda\mapsto \mathcal{A}(\lambda,\pm\infty)$ are continuous. In the following section we will explain the assumptions $(F2)$ and $(F3)$, and we will show that they are weaker than the required asymptotic hyperbolicity in \cite{SkibaIch}.

\section{Exponential dichotomy versus asymptotic hyperbolicity}\label{versus}

We begin this section by recalling the concept of an exponential dichotomy which was introduced by Palmer, Aulbach, Kalkbrenner and others. It was invented to extend the idea of hyperbolicity for autonomous discrete dynamical systems to explicitly non-autonomous non-invertible discrete dynamical systems. Below we will also discuss the relation to asymptotically hyperbolic systems that we considered in our previous work \cite{SkibaIch}.
We begin with the concept of an invariant projector.

\begin{defi}\label{ED-def}
A projector $\mathbb{P}\colon \mathbb{I}\to \mathcal{L}(X)$ is said to be {\it invariant} with respect to a discrete vector field $\mathbb{A}\colon \Z\to \mathcal{L}(X)$ if the diagram
\begin{align}\label{diagram-p}
\begin{gathered}
\xymatrix@C=40pt@R=20pt{
\cdots\ar[r] & \ar[d]^{\mathbb{P}_n} X \ar[r]^{\mathbb{A}_n}& X\ar[d]^{\mathbb{P}_{n+1}} \ar[r]^{\mathbb{A}_{n+1}} & X \ar[d]^{\mathbb{P}_{n+2}} \ar[r] & \cdots
\\
\cdots\ar[r] &  X\ar[r]^{\mathbb{A}_n}& X \ar[r]^{\mathbb{A}_{n+1}} & X \ar[r]& \cdots
 }
 \end{gathered}
\end{align}
is commutative and $\sup_{n\in \mathbb{I}}\|\mathbb{P}_n\|<\infty$.
\end{defi}
Note that  \eqref{diagram-p} implies that the diagram
\begin{align*}
\xymatrix@C=55pt@R=20pt{
\cdots\ar[r] & \ar[d]_{I-\mathbb{P}_n} X\ar[r]^{\mathbb{A}_n}& \ar[d]_{I-\mathbb{P}_{n+1}} X \ar[r]^{\mathbb{A}_{n+1}} & \ar[d]_{I-\mathbb{P}_{n+2}}X \ar[r]& \cdots \\
 \cdots\ar[r] & \ker \mathbb{P}_n \ar[r]^-{\mathbb{A}_n|\ker \mathbb{P}_n}& \ker \mathbb{P}_{n+1}\ar[r]^-{\mathbb{A}_{n+1}|\ker \mathbb{P}_{n+1}} & \ker \mathbb{P}_{n+2} \ar[r]&\cdots
 }
\end{align*}
commutes as well.

\begin{defi}
We say that an invariant projector $\mathbb{P}\colon \mathbb{I}\to \mathcal{L}(X)$ is regular $($with respect to $\mathbb{A})$ provided
$\mathbb{A}_n|\ker \mathbb{P}_n\colon \ker \mathbb{P}_n\to \ker \mathbb{P}_{n+1}$ is an isomorphism for all $n\in \N$ with $n,n+1\in \mathbb{I}$.
\end{defi}
\noindent
Note that we can always define a projector $\mathbb{P}\colon \mathbb{I}\to \mathcal{L}(X)$
by $\mathbb{P}_n=I_d$ for all $n\in \mathbb{I}$ or $\mathbb{P}_n=0$ for all $n\in \mathbb{I}$. Both projections will be called trivial.
\begin{rem}
It is readily seen that if $\mathbb{A}\colon \Z\to GL(X)$, then any invariant projector  $\mathbb{P}\colon \mathbb{I}\to \mathcal{L}(X)$ with respect to $\mathbb{A}$ is regular.
\end{rem}
\noindent
We shall now recall some non-trivial examples to illustrate that invariant projectors appear quite naturally in our setting.
\begin{ex}[Isomorphic discrete systems]
Assume that $\mathbb{A}\colon \Z\to GL(d,\R)$ and let $E_0$ be a subspace of $\R^d$. We consider the sequence of subspaces $(E_n)$ given by $E_{n+1}:=\mathbb{A}_n(E_{n})$ and $E_{-n-1}:=\mathbb{A}^{-1}_{-n-1}(E_{-n})$ for all $n\geqslant 0$. Then we can define a projector $\mathbb{P}\colon \Z\to \mathcal{L}(\R^d)$ by
\begin{equation*}
\mathbb{P}_n(x)=\left\{
\begin{array}{ll}
x & \emph{ if }\; x\in E_n, \\
0 & \emph{ if }\; x\in E_n^{\perp},
\end{array}
\right.
\end{equation*}
where $E_n^{\perp}$stands for the orthogonal complement of $E_n$ in $\R^d$. It is easy to see that $\mathbb{A}\colon \Z\to \mathcal{L}(\R^d)$ is regularly $P$-invariant. It is clear that the above projector $\mathbb{P}\colon \Z\to \mathcal{L}(\R^d)$ can be considered on any discrete interval $\mathbb{I}$.
\end{ex}
\noindent
The above example can be extended as follows.
\begin{ex}[Asymptotically isomorphic systems]\label{Example} We say that a linear discrete dynamical system $\mathbb{A}\colon \Z\to \mathcal{L}(\R^d)$ is asymptotically
isomorphic provided the two limits
\begin{equation*}
\lim_{n\to\pm\infty}\mathbb{A}_n=\mathbb{A}(\pm\infty)\in GL(d,\R)
\end{equation*}
exist. Then the system $\mathbb{A}\colon \Z\to \mathcal{L}(\R^d)$ is $P$-invariant
on both $\Z^-_{\underline{\kappa}}$ and $\Z^+_{\overline{\kappa}}$ for some integers $\underline{\kappa}<\overline{\kappa}$ as follows. Since $\mathbb{A}(\infty)\in GL(d,\R)$ $($resp. $\mathbb{A}(-\infty)\in GL(d,\R))$, there exists $\overline{\kappa}\in \Z$ such that $\mathbb{A}(k)\in GL(d,\R)$ for all $k\geqslant \overline{\kappa}$ $($resp. $\mathbb{A}(k)\in GL(d,\R)$ for all $k\leqslant \underline{\kappa})$. Thus one can construct
a projector $\mathbb{P}_+$ on the set $\Z_{\overline{\kappa}}$ $($resp. $\mathbb{P}_-$ on the set $\Z_{\underline{\kappa}})$ as in the previous example.
\end{ex}
\noindent
A discrete vector field $\mathbb{A}\colon \Z\to \mathcal{L}(\mathbb{R}^d)$  induces for $\mathbb{I}\subset \Z$ an operator $\Phi_{\mathbb{A}} \colon \mathbb{I}\boxtimes \mathbb{I}\to \mathcal{L}(X)$ by
\begin{equation*}
\Phi_{\mathbb{A}}(k,n):=\left\{
\begin{array}{ll}
\mathbb{A}_{k-1}\circ \mathbb{A}_{k-2}\circ\cdots\circ \mathbb{A}_{n} & \text{ if }\; k>n,\,k,n\in \mathbb{I}, \\
\id & \text{ if }\; k=n,\, k\in \mathbb{I},
\end{array}
\right.
\end{equation*}
where $\mathbb{I}\boxtimes \mathbb{I}:=\{(k,l)\in \mathbb{I}\times \mathbb{I}\mid l\leqslant k\}$. In what follows, we will usually omit the symbol index $\mathbb{A}$ from $\Phi_{\mathbb{A}}$. Now we are ready to introduce the concept of an exponential dichotomy with the aim to generalize the notion of hyperbolicity.

\begin{defi}[Exponential dichotomy on $\mathbb{I}$]\label{ED-ED} A linear discrete vector field $\mathbb{A}\colon \Z\to \mathcal{L}(\R^d)$
is said to admit an exponential dichotomy $($for short, ED$)$ on $\mathbb{I}$ if there exists an invariant regular projector $\mathbb{P}\colon \mathbb{I}\to \mathcal{L}(\R^d)$ and real numbers $K\geqslant 1$, $\alpha\in (0,1)$ such that for $k\geqslant n$ with $k,n\in \mathbb{I}$, $x\in \R^d$
\begin{align}
|\Phi(k,n)\mathbb{P}(n)x| &\leqslant K\alpha^{k-n}|\mathbb{P}(n)x|\label{ED1},\\
|\Phi(k,n)(I_d-\mathbb{P}(n))x| &\geqslant (1/K)(1/\alpha)^{k-n}|(I_d-\mathbb{P}(n))x|. \label{ED2}
\end{align}
\end{defi}
\noindent
In order to obtain interesting examples of projections yielding an ED for a discrete vector field $\mathbb{A}\colon \Z\to \mathcal{L}(\R^d)$, we first need to recall some facts about spectral projections for hyperbolic matrices. For any hyperbolic matrix $\mathcal{A}\in M(d,\mathbb{C})$, the spectral projection $P_{\!\mathcal{A}}\colon \mathbb{C}^d\to \mathbb{C}^d$ is defined by
\begin{align}\label{comP}
P_{\!\mathcal{A}}=\frac{1}{2\pi i}\oint_{S^1}R(z,\mathcal{A})\, dz,
\end{align}
where $R(\cdot,\mathcal{A})\colon \C\setminus \sigma(A)\to \mathcal{L}(\C^d)$ is the resolvent of $\mathcal{A}$ given by $R(z,\mathcal{A}):=(zI_d-\mathcal{A})^{-1}$ for $z\in \C\setminus \sigma(\mathcal{A})$. Note that $P_{\!\mathcal{A}}$ and $\mathcal{A}$ commute. Let us recall the following important fact, which can be found, e.g., in \cite{Pituk}.

\begin{lemma}\label{projection-image}
If $\mathcal{A}\in H(d,\R)$, then $P_{\!\mathcal{A}}x\in \R^d$ for all $x\in \R^d$.
\end{lemma}
Thus for $\mathcal{A}\in H(d,\R)$ we can set
\begin{align}\label{reP}
P^s_{\!\mathcal{A}}x:=P_{\!\mathcal{A}}x\in M(d,\R),\quad P^u_{\!\mathcal{A}}x:=x-P^s_{\!\mathcal{A}}x\in M(d,\R),\quad x\in\R^d.
\end{align}
It is well-known that the image of $P^s_{\!\mathcal{A}}$ is the stable space $E^s$ of $\mathcal{A}$, i.e., the space of real parts of all generalized eigenvectors with respect to eigenvalues inside the unit circle $S^1$. Analogously, the image of $P^u_{\!\mathcal{A}}$ is the unstable space $E^u$ of $\mathcal{A}$, i.e., the space of real parts of all generalized eigenvectors having eigenvalues outside $S^1$. It is clear that $\ker P_{\!\mathcal{A}}^s=\im P_{\!\mathcal{A}}^u$. We obtain a first interesting example of an ED.

{\begin{ex}\label{ED=H}
We assume that $\mathbb{A}\colon \Z\to \mathcal{L}(\R^d)$ is autonomous and hyperbolic, i.e., $\mathbb{A}(n)=\mathcal{A}$ for all $n\in\Z$, where $\mathcal{A}$ is a hyperbolic matrix. Then $\Phi(k,n)=\mathcal{A}^{k-n}$, and we see that \eqref{ED1} and \eqref{ED2} for $\mathbb{I}=\Z$ and $\mathbb{P}(n)=P_{\!\mathcal{A}}^s$, $n\in\Z$,  are equivalent to
\begin{align*}
|\mathcal{A}^{k-n}x| &\leqslant K\alpha^{k-n}|x|, & \text{ for all }k\geqslant n \text{ with }k,n\in \Z,\;x\in E^s,\\
|\mathcal{A}^{k-n}y| &\geqslant (1/K)(1/\alpha)^{k-n}|y|, & \text{ for all }k\geqslant n \text{ with }k,n\in \Z,\;
y\in E^u.
\end{align*}
As this is satisfied for any hyperbolic matrix $\mathcal{A}$, and as $\mathbb{P}(n)=P_{\!\mathcal{A}}^s$, $n\in\Z$, is a regular invariant projector, we see that autonomous systems have an ED as long as they are hyperbolic.
\end{ex}}
\noindent
We now consider spectral projections for linear bounded operators $T\colon \ell^{\infty}(\mathbb{C}^d)\to \ell^{\infty}(\mathbb{C}^d)$ such that $\sigma(T)\cap S^1=\emptyset$. In this case the spectral projections are analogously defined by
\begin{align}
P_{T}=\frac{1}{2\pi i}\oint_{S^1}R(z,T)\, dz.
\end{align}
As $\sigma(T)\cap S^1=\emptyset$, it follows that
$
\sigma(T_1)\subset \{z\in \mathbb{C}\mid |z|<1\}$ and  $\sigma(T_2)\subset \{z\in \mathbb{C}\mid |z|>1\}$,
where
$T_1:=T|\im(P_{T})$ and $T_2:=T|\ker(P_{T})$ (see Theorem 2.2 in \cite[pp. 10-11]{Gohberg}).
Moreover,
$
T_2\colon \ker(P_T)\to \ker(P_T)$ is an isomorphism, and all these properties of course also hold for the finite-dimensional hyperbolic operators $T=\mathcal{A}\colon \mathbb{C}^d\to \mathbb{C}^d$.\\
For the following discussion, we need that the spectrum of the composition of a shift operator with any linear bounded operator on $\ell^{\infty}(\mathbb{C}^d)$ has a considerable property.
\begin{lemma}\label{AM-lemma}\emph{\cite[Theorem 1]{Aulbach1}}
If $L\colon \ell^{\infty}(\mathbb{C}^d)\to \ell^{\infty}(\mathbb{C}^d)$ is a linear bounded operator, then the spectrum $\sigma(S_r\circ L)$ of $S_r\circ L$ is
rotationally symmetric, i.e., $\exp(i\mu)\lambda\in \sigma(S_r\circ L)$ for all $\mu\in\R$ and $\lambda\in \sigma(S_r\circ L)$, where $S_r\colon \ell^{\infty}(\mathbb{C}^d)\to \ell^{\infty}(\mathbb{C}^d)$  is the right shift operator given by $(S_r\phi)(n):=\phi(n-1)$ for all $\phi\in \ell^{\infty}(\mathbb{C}^d)$.
\end{lemma}
\noindent
Let us now consider as in Example \ref{ED=H} the equation \eqref{linear-equation} with constant coefficients meaning that
$\mathbb{A}(n)=\mathcal{A}\in M(d,\R)$ is constant for all $n\in\Z$ and the spectrum $\sigma(\mathcal{A})$ of $\mathcal{A}$ does not intersect the unit circle. Note that we have already seen in Example \ref{ED=H} that \eqref{linear-equation} has an ED with respect to the projection on the stable subspace of $\mathcal{A}$. It is well-known that an autonomous hyperbolic system
\begin{equation}\label{autonom-A}
\phi(n+1)=\mathcal{A}\phi(n),
\end{equation}
where $\phi\colon \Z\to \R^d$ is bounded, admits only the trivial solution in the space $\ell^{\infty}(\R^d)$. This implies that $1\not\in \sigma(S_r\circ \mathcal{N}_{\mathcal{A}})$, where $\mathcal{N}_{\mathcal{A}}\colon \ell^{\infty}(\R^d)\to \ell^{\infty}(\R^d)$ is the substitution operator given by $(\mathcal{N}_{\mathcal{A}}\phi)(n):=\mathcal{A}\phi(n)$, for all $\phi\in \ell^{\infty}(\R^d)$. Consequently, taking into account Lemma \ref{AM-lemma}, we can conclude that if $\mathcal{A}$ is hyperbolic, then
\begin{equation}\label{spectrum-nemitski}
\sigma(S_r\circ \mathcal{N}_{\mathcal{A}})\cap S^1=\emptyset.
\end{equation}
Hence $S_r\circ\mathcal{N}_{\mathcal{A}}$ induces a spectral projection $P_{S_r\circ\mathcal{N}_{\!\mathcal{A}}}\colon \ell^{\infty}(\mathbb{C}^d)\to \ell^{\infty}(\mathbb{C}^d)$ by
\begin{align*}
P_{S_r\circ\mathcal{N}_{\!\mathcal{A}}}=\frac{1}{2\pi i}\oint_{S^1}R(z,S_r\circ\mathcal{N}_{\mathcal{A}})\, dz.
\end{align*}
Our next aim is to show that there exists a connection between $P_{S_r\circ\mathcal{N}_{\!\mathcal{A}}}\colon \ell^{\infty}(\mathbb{C}^d)\to \ell^{\infty}(\mathbb{C}^d)$ and $P_{\!\mathcal{A}}\colon \mathbb{C}^d\to \mathbb{C}^d$.
Observe that $P_{S_r\circ\mathcal{N}_{\mathcal{A}}}$ allows us to define
a linear operator $\mathbb{P}_{\!\mathcal{A}}^s\colon\Z\to \mathcal{L}(\R^d)$ by
\begin{equation}\label{projection-P2}
\mathbb{P}_{\!\mathcal{A}}^s(n)x:=(P_{S_r\circ\mathcal{N}_{\mathcal{A}}}\phi_{n,x})(n),
\end{equation}
for all  $x\in \R^d,\;n\in\Z$, where $\phi_{n,x}\colon \Z\to \R^d$ is given by
$
\phi_{n,x}(m):=
\mathbbm{1}_{\{m\}}(n)x.
$
It is easy to see that $\mathbb{P}_{\!\mathcal{A}}^s\colon\Z\to \mathcal{L}(\R^d)$ is a projector (see \cite[p. 255]{Aulbach1}),
and the next result states that both constructions \eqref{reP} and \eqref{projection-P2} coincide. Note that this implies by Example \ref{ED=H} that \eqref{autonom-A} also has an ED with respect to $\mathbb{P}_{\!\mathcal{A}}^s(n)$, $n\in\mathbb{Z}$.

\begin{prop}\label{prop-projconstant} If $\mathcal{A}\in H(d,\R)$, then
$\mathbb{P}_{\!\mathcal{A}}^s(n)=P_{\!\mathcal{A}}^s$, for all $n\in\Z$.
\end{prop}
\begin{proof}
{It follows from \cite{Aulbach0,Aulbach1} that \eqref{autonom-A} has an ED with respect to $\mathbb{P}_{\!\mathcal{A}}^s(n)$, $n\in\mathbb{Z}$ (c.f. Proposition \ref{projection-A-B} below). Now, if \eqref{autonom-A} has an ED on $\Z$ with respect to two projections $P_1, P_2:\Z\rightarrow\mathcal{L}(\R^d)$, then $P_1$ and $P_2$ coincide (see \cite[Lem. 2.1 and Lem. 2.9]{Russ} or \cite[Rem. 3.4.17 and Rem. 3.4.18]{Christian10a})}.
\end{proof}
\noindent
Our next aim is to show that the projection from \eqref{projection-P2} for a single hyperbolic matrix $\mathcal{A}$ can be extended to the non-autonomous case. For this purpose, assume that the linear non-autonomous difference equation
\begin{equation*}\label{Nautonom-A}
\phi(n+1)=\mathbb{A}_n\phi(n), \quad n\in\Z,
\end{equation*}
admits a spectral dichotomy with respect to $\ell^{\infty}(\R^d)$, i.e., the spectrum $\sigma(S_r\circ \mathcal{N}_{\mathbb{A}})$ of $S_r\circ \mathcal{N}_{\mathbb{A}}$ does not intersect the unit circle in the complex plane, where $\mathcal{N}_{\mathbb{A}}\colon \ell^{\infty}(\R^d)\to \ell^{\infty}(\R^d)$  is defined by
$(\mathcal{N}_{\mathbb{A}}\phi)(n):=\mathbb{A}_n\phi(n)$ for all $\phi\in \ell^{\infty}(\R^d)$ and $n\in\Z$. Then $\sigma(S_r\circ\mathcal{N}_{\mathbb{A}})\cap S^1=\emptyset$ by Lemma \ref{AM-lemma}, which implies that the spectral projection
$P_{S_r\circ\mathcal{N}_{\mathbb{A}}}\colon \ell^{\infty}(\mathbb{C}^d)\to \ell^{\infty}(\mathbb{C}^d)$
\begin{align*}
P_{S_r\circ\mathcal{N}_{\mathbb{A}}}=\frac{1}{2\pi i}\oint_{S^1}R(z,S_r\circ\mathcal{N}_{\mathbb{A}})\, dz
\end{align*}
is defined. Moreover, $
P_{S_r\circ\mathcal{N}_{\mathbb{A}}}(\ell^{\infty}(\R^d))\subset \ell^{\infty}(\R^d)$,
which can be seen as in the proof Lemma \ref{projection-image}. The above projection $P_{S_r\circ\mathcal{N}_{\mathbb{A}}}$ allows us to define the family of operators
$P_{\!\mathbb{A}}^s\colon \Z\to \mathcal{L}(\R^d)$ by
\begin{equation}\label{operator}
P_{\mathbb{A}}^s(n)x:=(P_{S_r\circ\mathcal{N}_{\mathbb{A}}}\phi_{n,x})(n)
\end{equation}
for all $n\in\Z$ and $x\in\R^d$, where $\phi_{n,x}$ is defined as before. The following remarkable result shows that this projection is another example of an exponential dichotomy.
{\begin{prop}[\cite{Aulbach0,Aulbach1}]\label{projection-A-B}
Let $P_{\mathbb{A}}^s\colon \Z\to \mathcal{L}(\R^d)$ be
defined as in \eqref{operator}, where $\mathbb{A}\colon \Z\to \mathcal{L}(\R^d)$ is a discrete vector field admitting an ED on $\mathbb{Z}$. Then
\begin{itemize}
\item $P_{\mathbb{A}}^s$ is a regular projector commuting with $\mathbb{A}$, i.e.,
$\mathbb{A}_n\circ P_{\mathbb{A}}^s(n)=P_{\mathbb{A}}^s(n+1)\circ \mathbb{A}_n$, for all $n\in\Z$,
\item $\mathbb{A}$ admits an exponential dichotomy on $\Z$ with respect to the projector $P_{\mathbb{A}}^s$.
\end{itemize}
\end{prop}}


Assume that the invariant projector $\mathbb{P}\colon \mathbb{I}\to \mathcal{L}(\R^d)$ is regular, and $k\geqslant n$ with $k,n\in \mathbb{I}$. Since $\Phi(k,n)|\ker(\mathbb{P}(n))\colon \ker(\mathbb{P}(n))\to \ker(\mathbb{P}(k))$ is an isomorphism for all $k\geqslant n$ with $k,l\in \mathbb{I}$, we can define $\Phi(n,k)\colon \ker(\mathbb{P}(k))\to \ker(\mathbb{P}(n))$ as the inverse of $\Phi(k,n)|\ker(\mathbb{P}(n))$. Moreover, one can show that in this case \eqref{ED2} is equivalent to
\begin{equation}\label{ED3}
|\Phi(n,k)(I_d-\mathbb{P}(k))x|\leqslant K\alpha^{k-n}|(I_d-\mathbb{P}(k))x| \text{ for all }x\in\R^d.
\end{equation}
As $\im(I_d-\mathbb{P}(m))=\ker(\mathbb{P}(m))$, this allows to define \textit{Green's function}, which is the map $G_\Phi:\mathbb{I}\times\mathbb{I}\rightarrow\mathcal{L}(\mathbb{R}^d)$ given by
\begin{equation}\label{Green}
\mathbb{G}_{\Phi}(n,m)=\left\{
\begin{array}{ll}
\Phi(n,m)\mathbb{P}(m) & \emph{ if }\; m\leqslant n, \\
-\Phi(n,m)(I_d-\mathbb{P}(m)) & \emph{ if }\; n<m.
\end{array}
\right.
\end{equation}
We note for later reference the following result about \eqref{ED1}.
\begin{lemma}\label{ED-inverse}
Assume that the linear discrete vector field $\mathbb{A}\colon \Z\to \mathcal{L}(\R^d)$ has an invariant regular projector $\mathbb{P}\colon \mathbb{I}\to \mathcal{L}(\R^d)$. Then
\begin{align*}
|\Phi(k,n)\mathbb{P}(n)x|\leqslant K\alpha^{k-n}|\mathbb{P}(n)x| \Longleftrightarrow |||\Phi(k,n)^{-1}(\mathbb{P}(k)x)|||\geqslant (1/K)(1/\alpha)^{k-n}|\mathbb{P}(k)x|
\end{align*}
for all $k\geqslant n$ with $k,n\in \mathbb{I}$, where $|||\Phi(k,n)^{-1}y|||:=\inf\{|z|\mid z\in \Phi(k,n)^{-1}y\}$ and $K\geqslant 1$.
\end{lemma}
\begin{proof}
For showing the implication from left to right, we take $x\in \R^d$ such that $\mathbb{P}(k)x\neq 0$ and $z\in \Phi(k,n)^{-1}(\mathbb{P}(k)x)$. Then $\Phi(k,n)z=\mathbb{P}(k)x$ and $z\in \im(\mathbb{P}(n))$, where the latter holds as $\Phi(k,n)(\im(\mathbb{P}(n)))\subset \im(\mathbb{P}(k))$ and $\Phi(k,n)|\ker(\mathbb{P}(n)))\colon \ker(\mathbb{P}(n)))\to \ker(\mathbb{P}(k))$ is an isomorphism. Thus
\begin{align*}
1=|\mathbb{P}(k)x|^{-1}|\mathbb{P}(k)x|=|\mathbb{P}(k)x|^{-1}|\Phi(k,n)z|\leqslant |\mathbb{P}(k)x|^{-1}K\alpha^{k-n}|z|
\end{align*}
and hence $
|z|\geqslant (1/K)(1/\alpha)^{k-n}|\mathbb{P}(k)x|$,
which implies that
\begin{align*}
|||\Phi(k,n)^{-1}(\mathbb{P}(k)x)|||=\inf\{|z|\mid z\in \Phi(k,n)^{-1}(\mathbb{P}(k)x)\} \geqslant (1/K)(1/\alpha)^{k-n}|\mathbb{P}(k)x|.
\end{align*}
To show the implication from right to left, we note that since $\mathbb{P}(n)x\in \Phi(k,n)^{-1}(\Phi(k,n)\mathbb{P}(n)x)$ and $\Phi(k,n)\mathbb{P}(n)x=\mathbb{P}(k)\Phi(k,n)x$, it follows that
\begin{align*}
|\mathbb{P}(n)x|\geqslant |||\Phi(k,n)^{-1}(\Phi(k,n)\mathbb{P}(n)x)|||\geqslant (1/K)(1/\alpha)^{k-n}|\Phi(k,n)\mathbb{P}(n)x)|.
\end{align*}
Hence
$
K\alpha^{k-n}|\mathbb{P}(n)x|\geqslant |\Phi(k,n)\mathbb{P}(n)x|$,
which completes the proof.
\end{proof}
\noindent
Next we introduce the dichotomy spectrum of $\mathbb{A}\colon \Z\to \mathcal{L}(\R^d)$ which is strictly related to an ED. We set
\begin{itemize}
\item $\Sigma(\mathbb{A}):=\{\gamma>0\mid \gamma^{-1}\mathbb{A} \text{ does not have an ED on }\Z\}$,
\item $\Sigma_{\kappa}^{\pm}(\mathbb{A}):=\{\gamma>0\mid \gamma^{-1}\mathbb{A}\text{ does not have an ED on }\Z_{\kappa}^{\pm}\}$,
\end{itemize}
where $\gamma^{-1}\mathbb{A}\colon \to \mathcal{L}(\R^d)$ is given by $(\gamma^{-1}\mathbb{A})(n)=\gamma^{-1}\cdot \mathbb{A}(n)$ for $n\in\Z$.
The following result from \cite{Aulbach1, Poetzschec} summarizes the main properties of $\Sigma(\mathbb{A})$ and shows that it is closely related to $\sigma(S_r\circ \mathcal{N}_{\mathbb{A}})$.
\begin{theorem}\label{ED-ES}
\begin{enumerate}
\item[$(1)$] $\mathbb{A}$ has an ED on $\Z \Longleftrightarrow1\not\in \Sigma(\mathbb{A})$.
\item[$(2)$] $\Sigma(\mathbb{A})=\sigma(S_r\circ \mathcal{N}_{\mathbb{A}})\cap (0,\infty)$, and hence $\sigma(S_r\circ \mathcal{N}_{\mathbb{A}})\cap S^1=\emptyset$ if and only if $1\not\in\Sigma(\mathbb{A})$ (cf. Lemma \ref{AM-lemma}).
\item[$(3)$] If $1\not\in \Sigma(\mathbb{A})$, then there exists $\gamma_{\mathbb{A}}>0$ such that $1\not\in \Sigma(\mathbb{A}+\mathbb{B})$ for any $\mathbb{B}\colon\Z\to \mathcal{L}(\R^d)$ with $\|\mathbb{B}(n)\|\leqslant \gamma_{\mathbb{A}}$, for all $n\in\Z$.
\item[$(4)$] If $\mathbb{A}$ is autonomous, i.e., there exists $\mathcal{A}\in M(d,\R)$ such that $\mathbb{A}(n)=\mathcal{A}$ for all $n\in\Z$, then
\begin{align*}
\Sigma(\mathcal{A}):=\Sigma(\mathbb{A})=\{|\lambda|\mid \lambda\in\sigma(\mathcal{A})\}\setminus \{0\}.
\end{align*}
\end{enumerate}
\end{theorem}
Let us note that the projection yielding an ED in $(2)$ is the one from Proposition {\ref{projection-A-B}}.\\
The previous theorem suggests the following definitions. Let $\mathbb{A}\colon \Lambda\times\Z\to \mathcal{L}(\R^d)$ be a linear parametrized discrete vector field. We will say that
\begin{enumerate}
\item[(a)] $\mathbb{A}$ is hyperbolic provided $1\not\in\Sigma(\mathbb{A}(\lambda))$ for all $\lambda\in\Lambda$,
\item[(b)] given $\gamma>0$, $\mathbb{B}\colon \Lambda\times\Z\to \mathcal{L}(\R^d)$ is $\gamma$--small if
$
\|\mathbb{B}(\lambda,n)\|\leqslant \gamma$ for all $n\in\Z$ and $\lambda\in\Lambda$.
\item[(c)] given $\gamma_{\pm}>0$, $\mathbb{B}_{\pm}\colon \Lambda\times\Z\to \mathcal{L}(\R^d)$ is $\gamma_{\pm}$--small at $\pm\infty$ if there exists $\kappa_{\pm}\in \Z^{\pm}$ such that
\begin{align*}
\|\mathbb{B}_{\pm}(\lambda,n)\|\leqslant \gamma_{\pm} \text{ for all }n\in \Z^{\pm}_{\kappa_{\pm}}\text{ and }\lambda\in\Lambda.
\end{align*}
\end{enumerate}
Let us recall from Example \ref{ED=H} that \eqref{autonom-A} has an ED which is given by the projection on the stable space of $\mathcal{A}$. Note that if $\mathcal{A}\colon \Lambda \to H(d,\R)$ is a continuous family of hyperbolic matrices, then the corresponding family of projections \eqref{reP} is continuous as well. The following perturbation theorem of P\"otzsche can be found in \cite{Poetzsched}.

\begin{prop}\label{family-hyperbolic}
For any continuous family of hyperbolic matrices $\mathcal{A}\colon \Lambda \to H(d,\R)$ there exists $\gamma_{\mathcal{A}}>0$ such that
if $\mathbb{B}\colon \Lambda\times\Z\to \mathcal{L}(\R^d)$ is $\gamma_{\mathcal{A}}$--small, then there is a parametrized projector $\mathbb{P}\colon \Lambda\times\Z \to \mathcal{L}(\R^d)$ such that
\begin{itemize}
\item $\mathbb{P}(\lambda)\colon \Z\to \mathcal{L}(\R^d)$ is regular and commutes with $\mathcal{A}(\lambda)+\mathbb{B}(\lambda)$ for all $\lambda\in\Lambda$,
\item $\mathcal{A}(\lambda)+\mathbb{B}(\lambda)\colon\Z\to \mathcal{L}(\R^d)$ admits an ED on $\Z$ with respect to $\mathbb{P}(\lambda)$  for all $\lambda\in\Lambda$.
\end{itemize}
\end{prop}
\noindent
As final result of this section, we use P\"otzsche's theorem to show that an ED extends the notion of an asymptotically hyperbolic system. A consequence of this observation will be that our index theorem in Section \ref{section-index} as well as our bifurcation theorems in Section \ref{section-bifurcation} generalize the corresponding results that we previously obtained in \cite{SkibaIch}.

\begin{theorem}\label{perturbed-ED}
Let $\mathbb{A}\colon \Lambda\times\Z\to \mathcal{L}(\R^d)$ be asymptotically hyperbolic. Then there exists $\gamma_{\pm}>0$ such that if $\mathbb{D}\colon \Lambda\times\Z\to \mathcal{L}(\R^d)$ is $\gamma_{\pm}$--small at $\pm\infty$, then for the perturbed system $\mathbb{A}+\mathbb{D}$ there is a parametrized projector $\mathbb{P}_{\pm}\colon \Lambda\times\Z^{\pm}_{\kappa_{\pm}} \to \mathcal{L}(\R^d)$ for some $\kappa_{\pm}\in\Z$ such that
\begin{itemize}
\item $\mathbb{P}_{\pm}(\lambda)\colon \Z^{\pm}_{\kappa_{\pm}}\to \mathcal{L}(\R^d)$ is regular and commutes with $\mathbb{A}(\lambda)+\mathbb{D}(\lambda)$ for all $\lambda\in\Lambda$,
\item $\mathbb{A}(\lambda)+\mathbb{D}(\lambda)\colon\Z^{\pm}_{\kappa_{\pm}}\to \mathcal{L}(\R^d)$ admits an ED on $\Z_{\kappa_{\pm}}^{\pm}$ with respect to $\mathbb{P}_{\pm}(\lambda)$ for all $\lambda\in\Lambda$.
\end{itemize}
\end{theorem}

\begin{proof}
We recalled in the final paragraph of Section \ref{section-Nemitski} that, if $\mathbb{A}$ is asymptotically hyperbolic, then there are two continuous families of hyperbolic matrices $\mathcal{A}_{\pm}\colon \Lambda\to H(d,\R)$ such that
\begin{equation*}
\mathcal{A}_{\pm}(\lambda)=\lim\limits_{n\to\pm\infty}\mathbb{A}(\lambda,n)
\end{equation*}
uniformly with respect to the parameter $\lambda\in\Lambda$. Let now $\gamma_{\mathcal{A_\pm}}>0$ be as in Proposition \ref{family-hyperbolic} for $\mathcal{A}_{\pm}$.
Our aim is to show that the assertion of our theorem holds for $\gamma_{\pm}:=\gamma_{\mathcal{A}_{\pm}}/2$. Indeed, let
$\mathbb{D}\colon \Lambda\times\Z\to \mathcal{L}(\R^d)$ be $\gamma_+$--small at $+\infty$ (the case of $\mathbb{D}\colon \Lambda\times\Z\to \mathcal{L}(\R^d)$ being $\gamma_-\!$--small at $-\infty$ is similar and therefore it is left to the reader).
Then there exists $n_{+}>0$ such that for all $\lambda\in\Lambda$
\begin{align*}
\|\mathbb{A}(\lambda,n)-\mathcal{A}_{+}(\lambda)\|<\gamma_+ \text{ for }n\geqslant n_{+}.
\end{align*}
We define $\mathbb{B}_+\colon \Lambda\times\Z\to \mathcal{L}(\R^d)$ by
\begin{align*}
\mathbb{B}_{+}(\lambda,n):=\mathbbm{1}_{[n_+,\infty)}(n)(
\mathbb{A}(\lambda,n)-\mathcal{A}_+(\lambda)+\mathbb{D}(\lambda,n)).
\end{align*}
It is clear that $\mathbb{B}_+$ is $\gamma_{\mathcal{A}_+}\!$--small. Consequently, Proposition \ref{family-hyperbolic} implies the existence
of a parametrized projector $\widetilde{\mathbb{P}}_+\colon \Lambda\times\Z \to \mathcal{L}(\R^d)$ such that
\begin{itemize}
\item $\widetilde{\mathbb{P}}_+(\lambda)\colon \Z\to \mathcal{L}(\R^d)$ is regular and commutes with $\mathcal{A}_+(\lambda)+\mathbb{B}_+(\lambda)$ for all $\lambda\in\Lambda$,
\item $\mathcal{A}_+(\lambda)+\mathbb{B}_+(\lambda)\colon\Z\to \mathcal{L}(\R^d)$ admits an ED on $\Z$ with respect to $\widetilde{\mathbb{P}}_+(\lambda)$ for all $\lambda\in\Lambda$.
\end{itemize}
We set $\kappa_+:=n_+$ and note that the restrictions $\mathbb{P}_+(\lambda)\colon \Z^+_{\kappa_+}\to \mathcal{L}(\R^d)$ of $\widetilde{\mathbb{P}}_+(\lambda)\colon \Z\to \mathcal{L}(\R^d)$ to $\Z_{\kappa_+}^+$ are projections with the desired properties. This completes the proof.
\end{proof}

\section{Fredholm properties of Nemitski operators and a family index theorem}\label{section-index}

The aim of this section is to discuss Fredholm properties of certain functional operators\linebreak $L\colon\ell_0(\R^d)$ $\to \ell_0(\R^d)$ which are induced by linear discrete vector fields $\mathbb{A}\colon \Z\to \mathcal{L}(\R^d)$. We give a new proof of a well known index theorem, which we afterwards generalize to families of operators and where the indices are $KO$-theory classes following a construction of Atiyah and J\"anich. This latter result generalizes the main theorem of \cite{SkibaIch} for asymptotically hyperbolic systems to systems which merely have an ED.\\
We begin by showing a technical property which is of independent interest.

\begin{lemma}[Silverman-Toeplitz Theorem for sequences of matrices]\label{Sil-Toep} Assume that $f\colon \Z\times\Z\to M(d,\R)$ satisfies the following conditions:
\begin{enumerate}
\item[$(1)$] $f(n,\cdot)\in \ell_1(M(d,\R))$ for all $n\in \Z$ and $\sup\limits_{n\in\Z}\|f(n,\cdot)\|_{\ell_1(M(d,\R))}<\infty$,
\item[$(2)$] $f(\cdot,k)\in \ell_0(M(d,\R))$ for all $k\in\Z$.
\end{enumerate}
Then $f\ast \phi\in \ell^{\infty}(\R^d)$
$($resp. $f\ast \phi\in \ell_0^{\pm}(\R^d))$ if $\phi\in \ell^{\infty}(\R^d)$ $($resp. $\phi\in \ell_0^{\pm}(\R^d))$, where $f\ast \phi\colon \Z\to \R^d$ is defined by
\begin{align*}
(f\ast \phi)(n):=\sum_{k\in\Z} f(n,k)\phi(k) \text{ for all }n\in \Z.
\end{align*}
\end{lemma}
\begin{proof}
Throughout the proof we set $c_0:=\sup\limits_{n\in\Z}\|f(n,\cdot)\|_{\ell_1(M(d,\R))}$ which is finite by assumption.\\
We first assume that $\phi\in \ell^{\infty}(\R^d)$. Then
\begin{align*}
|(f\ast \phi)(n)|=&\left|\sum_{k\in\Z} f(n,k)\phi(k)\right|\leqslant \sum_{k\in\Z} |f(n,k)\phi(k)|\leqslant \\
&\sum_{k\in\Z} \|f(n,k)\|\cdot |\phi(k)|\leqslant \sum_{k\in\Z} \|f(n,k)\|\cdot \|\phi\|_{\infty}=\\
&\left(\sum_{k\in\Z} \|f(n,k)\|\right)\cdot\|\phi\|_{\infty}=\|f(n,\cdot)\|_{\ell_1(M(d,\R))}\cdot \|\phi\|_{\infty}\leq c_0\|\phi\|_{\infty}
\end{align*}
and hence $f\ast \phi\in\ell^{\infty}(\R^d)$.\\
Now assume that $\phi\in \ell_0^+(\R^d)$. We aim to show that $f\ast\phi\in \ell_0^+(\R^d)$. For any $m\in\N$ we set
\begin{align*}
\mathcal{D}_m:=\{(i,k)\mid i\geqslant m \text{ and }k\geqslant m \},\quad \mathcal{D}_{\!m}^c:=\N\times\N-\mathcal{D}_m
\end{align*}
and
\begin{equation*}
f^m(n,k):=\mathbbm{1}_{\mathcal{D}_{\!m}^c}(n,k)f(n,k).
\end{equation*}
Then, for $n\geqslant m$,
\begin{align*}
|(f^m\ast \phi)(n)|=&\left|\sum_{k\in\Z}f^m(n,k)\phi(k)\right|=\left|\sum_{k\in\Z}\mathbbm{1}_{\mathcal{D}_{\!m}^c}(n,k)f(n,k)\phi(k)\right|=\\
&\left|\sum_{k=0}^{m}\mathbbm{1}_{\mathcal{D}_{\!m}^c}(n,k)f(n,k)\phi(k)\right|=\left|\sum_{k=0}^{m}f(n,k)\phi(k)\right|\leqslant\\
&\sum_{k=0}^{m}|f(n,k)\phi(k)|\leqslant \sum_{k=0}^{m}\|f(n,k)\|\cdot |\phi(k)|\leqslant\|\phi\|_{\infty}\left(\sum_{k=0}^{m}\|f(n,k)\|\right)
\xrightarrow[n\rightarrow\infty]{}0
\end{align*}
and $(f^m\ast \phi)(n)=0$ for all $n<0$. Thus
\begin{align*}
f^m\ast \phi\in \ell_0^+(\R^d) \text{ for all }  \phi\in \ell_0^+(\R^d),\;m\in\N.
\end{align*}
As $f\ast \phi-f^m\ast \phi=(f-f^m)\ast \phi$, we obtain

\begin{align*}
&|((f-f^m)\ast \phi)(n)|=\left|\sum_{k\in\Z}(f-f^m)(n,k)\phi(k)\right|=\left|\sum_{k\in\Z}(1-\mathbbm{1}_{\mathcal{D}_{m}^c}(n,k))f(n,k)\phi(k)\right|=\\
&\left|\sum_{k\in\Z}\mathbbm{1}_{\mathcal{D}_m}(n,k)f(n,k)\phi(k)\right|=\left|\sum_{k=m}^{\infty}f(n,k)\phi(k)\right|\leqslant
\sum_{k=m}^{\infty}|f(n,k)\phi(k)|\leqslant \\ &\left(\sum_{k=m}^{\infty}\|f(n,k)\|\right)\left(\sup_{k\geqslant m}|\phi(k)|\right)\leqslant
\|f(n,\cdot)\|_{\ell_1(M(d,\R))}\left(\sup_{k\geqslant m}|\phi(k)|\right)\leqslant c_0\left(\sup_{k\geqslant m}|\phi(k)|\right)
\end{align*}
and so
\begin{equation*}
\|(f-f^m)\ast \phi\|_{\infty}\leqslant c_0\left(\sup_{k\geqslant m}|\phi(k)|\right)
\xrightarrow[m\rightarrow\infty]{}0.
\end{equation*}
Consequently, we have proved that
\begin{equation*}
f^m\ast \phi\xrightarrow[m\rightarrow\infty]{} f\ast \phi \text{ in } \ell^{\infty}_+(\R^d).
\end{equation*}
As $\ell_0^+(\R^d)$ is a closed subspace of $\ell^{\infty}_+(\R^d)$ and $f^m\ast \phi\in \ell_0^+(\R^d)$, we obtain that $f\ast \phi\in \ell_0^+(\R^d)$.

As the proof for $\phi\in\ell_0^-(\R^d)$ is very similar to the case $\phi\in\ell_0^+(\R^d)$, we leave this remaining case to the reader. This completes the proof.
\end{proof}
Now we are beginning our discussion of Fredholm properties.

\begin{prop}\label{Fredholm-L}
Assume that the bounded linear discrete vector field $\mathbb{A}\colon \Z\to \mathcal{L}(\R^d)$ admits an ED on $\N$
$($with regular projector $\mathbb{P}\colon \N\to \mathcal{L}(\R^d))$. Then the operator
$L^+\colon \ell_+^{\infty}(\R^d)\to \ell_+^{\infty}(\R^d)$ given by
\begin{align}\label{operator-L}
(L^+\phi)(n):=\mathbbm{1}_{\N}(n)\cdot((\mathbb{S}_l- \mathbb{A})\phi)(n),
\end{align}
for $\phi\in \ell_+^{\infty}(\R^d)$ and $n\in\Z$, is surjective and
\begin{equation}\label{kernel-L}
\ker(L^+)=\{\phi\in \ell_0^+(\R^d)\mid \phi(n)=\Phi(n,0)x,\;n\in\N,\;x\in \im(\mathbb{P}(0))\}.
\end{equation}
Moreover, $L^+(\ell_0^+(\R^d))\subset\ell_0^+(\R^d)$ and $L^+|\ell_0^+(\R^d)\colon \ell_0^+(\R^d)\to \ell_0^+(\R^d)$ is surjective as well $($\footnote{This result in the autonomous case (i.e. when $\mathbb{A}(n)=\mathcal{A}$ for all $n\in \Z$, where $\mathcal{A}$ is a hyperbolic matrix) was proved in the paper \cite[Lemma 2.1]{Ab-Ma2}. Moreover, it should be noted that a similar result was also considered by Baskakov in the paper \cite{baskakov}. However, he considered the slightly modified operator $(L\phi)(n)=\phi(n)-\mathbb{A}(n)\phi(n-1)$, in contrast to the operator $L^+$ defined in \eqref{operator-L} which is harder to handle. }$)$.
\end{prop}

\begin{proof}
Let us first observe that $L^+\colon \ell_+^{\infty}(\R^d)\to \ell_+^{\infty}(\R^d)$ is well-defined, which follows from
\begin{align*}
|(L^+\phi)(n)|=&|\phi(n+1)-\mathbb{A}(n)\phi(n)|\leqslant |\phi(n+1)|+\|A(n)\||\phi(n)|\leqslant\\
& \|\phi\|_{\infty}+\left(\sup_{n\in\N}\|\mathbb{A}(n)\|\right)\|\phi\|_{\infty}<\infty.
\end{align*}
We now define $M\colon \ell_+^{\infty}(\R^d)\to \ell_+^{\infty}(\R^d)$ by

\begin{align}\label{function-M}
(M\phi)(n)&=
\sum\limits_{k=0}^{\infty}\mathbbm{1}_{\N}(n)\mathbb{G}_{\Phi}(n,k+1)\phi(k)=\\
&=\begin{cases}\notag
-\sum\limits_{k=0}^{\infty}\Phi(0,k+1)(I_d-\mathbb{P}(k+1))\phi(k) & \text{ if }\; n=0,\\
\sum\limits_{k=0}^{n-1}\Phi(n,k+1)\mathbb{P}(k+1)\phi(k)-\sum\limits_{k=n}^{\infty}\Phi(n,k+1)(I_d-\mathbb{P}(k+1))\phi(k) & \text{ if }\; n> 0,\\
 0 &\text{ if }\; n<0.
\end{cases}
\end{align}
where $\mathbb{G}_{\Phi}$ is Green's function from \eqref{Green}. Our first aim is to show that $M$ is well-defined.
For $n>0$ we have
\begin{align*}
|M\phi(n)|=\left|\sum\limits_{k=0}^{\infty}\mathbbm{1}_{\N}(n)\mathbb{G}_{\Phi}(n,k+1)\phi(k)\right|\leqslant
K\frac{1+\alpha}{1-\alpha}\left(1+\sup_{k\in\N}\|\mathbb{P}(k)\|\right)\|\phi\|_{\infty}
\end{align*}
since
\begin{align}
\begin{split}
\label{estimation}
&\left|\sum\limits_{k=0}^{\infty}\mathbbm{1}_{\N}(n)\mathbb{G}_{\Phi}(n,k+1)\phi(k)\right|\leqslant \sum\limits_{k=0}^{\infty}|\mathbbm{1}_{\N}(n)\mathbb{G}_{\Phi}(n,k+1)\phi(k)|=\\
&\sum\limits_{k=0}^{n-1}|\Phi(n,k+1)\mathbb{P}(k+1)\phi(k)|+\sum\limits_{k=n}^{\infty}|\Phi(n,k+1)(I_d-\mathbb{P}(k+1))\phi(k)|.
\end{split}
\end{align}
and, by using the exponential dichotomy,
\begin{align}\label{estimation3}
\begin{split}
&\eqref{estimation}=\sum\limits_{k=0}^{n-1}K\alpha^{n-k-1}|\mathbb{P}(k+1)\phi(k)|+\sum\limits_{k=n}^{\infty}K\alpha^{k+1-n}|(I_d-\mathbb{P}(k+1))\phi(k)|\leqslant \\
&\sum\limits_{k=0}^{n-1}K\alpha^{n-k-1}\left(\sup_{k\in\N}\|\mathbb{P}(k)\|\right)\|\phi\|_{\infty}+
\sum\limits_{k=n}^{\infty}K\alpha^{k+1-n}\left(1+\sup_{k\in\N}\|\mathbb{P}(k)\|\right)\|\phi\|_{\infty}\leqslant\\
&K\left(1+\sup_{k\in\N}\|\mathbb{P}(k)\|\right)\|\phi\|_{\infty}\left(\sum\limits_{i=0}^{\infty}\alpha^{i}+\sum\limits_{i=0}^{\infty}\alpha^{i+1}\right)
=K\frac{1+\alpha}{1-\alpha}\left(1+\sup_{k\in\N}\|\mathbb{P}(k)\|\right)\|\phi\|_{\infty}.
\end{split}
\end{align}

\noindent
Moreover, for $n=0$ we have
\begin{align*}
|M\phi(0)|=&\left|\sum\limits_{k=0}^{\infty}\mathbbm{1}_{\N}(0)\mathbb{G}_{\Phi}(0,k+1)\phi(k)\right|=
\left|\sum\limits_{k=0}^{\infty}-\Phi(0,k+1)(I_d-\mathbb{P}(k+1))\phi(k)\right|\leqslant \\
&\sum\limits_{k=0}^{\infty}|\Phi(0,k+1)(I_d-\mathbb{P}(k+1))\phi(k)|\leqslant
\sum_{k=0}^{\infty}K\alpha^{k+1}|(I_d-\mathbb{P}(k+1))\phi(k)|\leqslant \\
&
\left(\sum_{k=0}^{\infty}\alpha^{k+1}\right)K\left(1+\sup_{k\in\N}\|\mathbb{P}(k)\|\right)\|\phi|_{\infty}=
K\frac{\alpha}{1-\alpha}\left(1+\sup_{k\in\N}\|\mathbb{P}(k)\|\right)\|\phi\|_{\infty},
\end{align*}
and finally $|M\phi(n)|=0$ for $n<0$, which proves that the operator $M\colon \ell_+^{\infty}(\R^d)\to \ell_+^{\infty}(\R^d)$ is indeed well-defined.
\newline\indent
Next we aim to show that $(L^+M)\phi=\phi$. Observe that for $n=0$ we have

\begin{align*}
&(L^+M\phi)(0)=S_l(M\phi)(0)-\mathbb{A}(0)(M\phi)(0)=(M\phi)(1)-\mathbb{A}(0)(M\phi)(0)=\\
&\mathbb{P}(1)\phi(0)-\sum^\infty_{k=1}\Phi(1,k+1)(I_d-\mathbb{P}(k+1))\phi(k)+\sum^\infty_{k=0}\mathbb{A}(0)\Phi(0,k+1)(I_d-\mathbb{P}(k+1))\phi(k)=\\
&\mathbb{P}(1)\phi(0)-\sum^\infty_{k=1}\Phi(1,k+1)(I_d-\mathbb{P}(k+1))\phi(k)+\sum^\infty_{k=0}\Phi(1,k+1)(I_d-\mathbb{P}(k+1))\phi(k)=\\
&\mathbb{P}(1)\phi(0)+(I_d-\mathbb{P}(1))\phi(0)=\phi(0),
\end{align*}

and for $n>0$
\begin{align}
&(L^+M\phi)(n)=S_l(M\phi)(n)-\mathbb{A}(n)(M\phi)(n)=(M\phi)(n+1)-\mathbb{A}(n)(M\phi)(n)=\notag\\
\begin{split}\label{estimation2}
&\sum^{n}_{k=0} \Phi(n+1,k+1)\mathbb{P}(k+1)\phi(k)-\sum^\infty_{k=n+1}\Phi(n+1,k+1)(I_d-\mathbb{P}(k+1))\phi(k)-\\
&\sum^{n-1}_{k=0}\mathbb{A}(n)\Phi(n,k+1)\mathbb{P}(k+1)\phi(k)+\sum^\infty_{k=n}\mathbb{A}(n)\Phi(n,k+1)(I_d-\mathbb{P}(k+1))\phi(k).
\end{split}
\end{align}

\noindent As $\mathbb{A}(n)\Phi(n,k+1)=\Phi(n+1,k+1)$,
\begin{align*}
&\eqref{estimation2}=\sum^{n}_{k=0} \Phi(n+1,k+1)\mathbb{P}(k+1)\phi(k)-\sum^\infty_{k=n+1}\Phi(n+1,k+1)(I_d-\mathbb{P}(k+1))\phi(k)-\\
&\sum^{n-1}_{k=0}\Phi(n+1,k+1)\mathbb{P}(k+1)\phi(k)+\sum^\infty_{k=n}\Phi(n+1,k+1)(I_d-\mathbb{P}(k+1))\phi(k)=\\
&\mathbb{P}(n+1)\phi(n)+(I_d-\mathbb{P}(n+1))\phi(n)=\phi(n),
\end{align*}
as well as $(L^+M\phi)(n)=0$ for $n<0$. Thus $L$ is surjective.\\
Now we are going to describe the kernel of $L^+\colon \ell_+^{\infty}(\R^d)\to \ell_+^{\infty}(\R^d)$. Let $\phi\in \ker(L^+)$. Then
\begin{align*}
\phi(n)=\Phi(n,0)\phi(0)
\end{align*}
for all $n\geqslant 0$. Since $\R^d=\im(\mathbb{P}(0))\oplus \ker(\mathbb{P}(0))$, we need to consider the two cases $\phi(0)\in \im(\mathbb{P}(0))$ and $\phi(0)\in \ker(\mathbb{P}(0))$. If $\phi(0)\in \im(\mathbb{P}(0))$ with $\phi\in \ker(L)$, then
\begin{align*}
|\phi(n)|=|\Phi(n,0)\phi(0)|\leqslant K\alpha^{n}|\mathbb{P}(0)\phi(0)|=K\alpha^n|\phi(0)|\xrightarrow[n\rightarrow\infty]{}0,
\end{align*}
while if $\phi(0)\in \ker(\mathbb{P}(0))$, then we would get
\begin{equation*}
|\phi(n)|=|\Phi(n,0)\phi(0)|=|\Phi(n,0)(I_d-\mathbb{P}(0))\phi(0)|\geqslant (1/K)(1/\alpha)^{n}|(I_d-\mathbb{P}(0))\phi(0)|
\xrightarrow[n\rightarrow\infty]{}\infty,
\end{equation*}
which contradicts the fact that $\phi\in \ell_+^{\infty}(\R^d)$. Thus we see that if $\phi\in\ker(L^+)$, then $\phi(0)\in \im(\mathbb{P}(0))$. Hence \eqref{kernel-L} is shown. \newline\indent It is not hard to see that $L^+(\ell_0^+(\R^d))\subset \ell_0^+(\R^d)$. Indeed, let $\phi\in \ell_0^+(\R^d)$. Then, for $n\geqslant 0$,
\begin{align*}
|(L^+\phi)(n)|=&|\phi(n+1)-\mathbb{A}(n)\phi(n)|\leqslant |\phi(n+1)|+\|A(n)\||\phi(n)|\leqslant\\
& |\phi(n+1)|+\left(\sup_{n\in\N}\|\mathbb{A}(n)\|\right)|\phi(n)|\xrightarrow[n\rightarrow\infty]{}0.
\end{align*}
For the surjectivity of $L|\ell_0^+(\R^d)\colon \ell_0^+(\R^d)\to \ell_0^+(\R^d)$, it is enough to show that the operator $M$ defined in
\eqref{function-M} satisfies
\begin{align}\label{conlusion-M}
|(M\phi)(n)|\xrightarrow[n\rightarrow\infty]{}0
\end{align}
for all $\phi\in \ell_0^+(\R^d)$. Define $f\colon \Z\times\Z\to M(d,\mathbb{R})$ by
\begin{align*}
f(n,k):=\mathbbm{1}_{\N}(n)\mathbbm{1}_{\N}(k)\mathbb{G}_{\Phi}(n,k+1).
\end{align*}
It is easy to see that $f$ satisfies
\begin{enumerate}
\item[$(1)$] $f(n,\cdot)\in \ell_1(M(d,\R))$, for all $n\in \Z$, with $\sup\limits_{n\in\Z}\|f(n,\cdot)\|_{\ell_1(M(d,\R))}<\infty$,
\item[$(2)$] $f(\cdot,k)\in \ell_0(M(d,\R))$ for all $k\in\Z$.
\end{enumerate}
Indeed, taking into account \eqref{estimation}-\eqref{estimation3}, we have
\begin{align*}
&\|f(n,\cdot)\|_{\ell_1(M(d,\R))}=\sum_{k\in\Z}\|\mathbbm{1}_{\N}(n)\mathbbm{1}_{\N}(k)\mathbb{G}_{\Phi}(n,k+1)\|\leqslant
K\frac{1+\alpha}{1-\alpha}\left(1+\sup_{k\in\N}\|\mathbb{P}(k)\|\right),\text{ for }\;n\in \N,\\
&\|f(n,\cdot)\|_{\ell_1(M(d,\R))}=\sum_{k\in\Z}\|\mathbbm{1}_{\N}(n)\mathbbm{1}_{\N}(k)\mathbb{G}_{\Phi}(n,k+1)\|=\sum_{k\in \Z}0=0, \text{ for }\;n<0
\end{align*}
and
\begin{align*}
\|f(n,k)\|&=\|\mathbbm{1}_{\N}(n)\mathbbm{1}_{\N}(k)\mathbb{G}_{\Phi}(n,k+1)\|=\|\mathbb{G}_{\Phi}(n,k+1)\|
=\|\Phi(n,k+1)\mathbb{P}(k+1)\|\\ &\leqslant K\alpha^{n-k-1}\|\mathbb{P}(k+1)\|
\leqslant K\alpha^{n-k-1}\left(\sup_{k\in \N}\|\mathbb{P}(k+1)\|\right)\xrightarrow[n\rightarrow\infty]{}0, \text{ for }k\geqslant 0,n\geqslant k,\\
\|f(n,k)\|&=\|\mathbbm{1}_{\N}(n)\mathbbm{1}_{\N}(k)\mathbb{G}_{\Phi}(n,k+1)\|=0\xrightarrow[n\rightarrow-\infty]{}0, \text{ for }k\geqslant 0,n<0, \\
\|f(n,k)\|&=\|\mathbbm{1}_{\N}(n)\mathbbm{1}_{\N}(k)\mathbb{G}_{\Phi}(n,k+1)\|=0\xrightarrow[n\rightarrow\pm\infty]{}0, \text{ for }k<0,n\in \Z.
\end{align*}
Finally, since
\begin{equation*}
M\phi=f\ast\phi \text{ for all } \ell^+_0(\R^d),
\end{equation*}
\eqref{conlusion-M} follows from Lemma \ref{Sil-Toep}. This completes the proof.
\end{proof}
\begin{cor}\label{Fredholm-L+}
Assume that the bounded dynamical system $\mathbb{A}\colon \Z\to \mathcal{L}(\R^d)$ admits an ED on $\Z_{\kappa}^+$
with a regular projector $\mathbb{P}\colon \Z_{\kappa}^+\to \mathcal{L}(\R^d)$ for some $\kappa\geqslant 0$. Then the operator
$L_{\kappa}^+\colon \ell^+_{0\kappa}(\R^d)\to \ell^+_{0\kappa}(\R^d)$ given by
\begin{align*}
(L_{\kappa}^+\phi)(n):=\mathbbm{1}_{\Z_{\kappa}^+}(n)\cdot((S_l- \mathbb{A})\phi)(n)
\end{align*}
for all $\phi\in \ell^+_{0\kappa}(\R^d)$ and $n\in\Z$, is surjective and
\begin{align}\label{kernel-L+}
\ker(L_{\kappa}^+)=\{\phi\in \ell^+_{0\kappa}(\R^d)\mid \phi(n)=\Phi(n,\kappa)x,\;n\in\Z_{\kappa}^+,\;x\in \im(\mathbb{P}(\kappa))\}.
\end{align}
In particular, $\ker(L_{\kappa}^+)$ is isomorphic to $\im(\mathbb{P}(\kappa))$ and so $\dim \ker(L_{\kappa}^+)=\dim \im(\mathbb{P}(\kappa))$.
\end{cor}
\begin{proof}
First, observe that $\mathbb{A}_0\colon \Z\to \mathcal{L}(\R^d)$ given by $\mathbb{A}_0(n):=\mathbb{A}(n+\kappa)$, for $n\in \Z$, admits an ED on $\N$ with respect to
$\mathbb{P}_0\colon \N\to \mathcal{L}(\R^d)$ defined by $\mathbb{P}_0(n):=\mathbb{P}(n+\kappa)$, for $n\in\Z$. Consider the commutative diagram
\begin{align}\label{diagram-L}
\begin{split}
\xymatrix@C=25pt@R=25pt{
\ell^+_{0\kappa}(\R^d)\ar[r]^-{L_{\kappa}^+} \ar[d]_-{\mathfrak{I}}^{\simeq} & \ell^+_{0\kappa}(\R^d) \ar[d]^-{\mathfrak{I}}_{\simeq}
\\
 \ell_0^+(\R^d)\ar[r]^-{L_0^+}   &
 \ell_0^+(\R^d),
 }
\end{split}
\end{align}
where $L_0^+$ and $\mathfrak{J}$ are defined by
\begin{align*}
(L_0^+\phi)(n)&=\mathbbm{1}_{\N}(n)\cdot((S_l- \mathbb{A}_0)\phi)(n)\text{ and }
(\mathfrak{J}\phi)(n)=\phi(n+\kappa),
\end{align*}
for $n\in\Z$. Now Proposition \ref{Fredholm-L} implies that $L_0^+$ is surjective with kernel as in \eqref{kernel-L}. Since $\mathfrak{J}$ is an isomorphism, it follows from the commutativity of the diagram \eqref{diagram-L} that $L_{\kappa}^+$ has the required properties. Finally, \eqref{kernel-L+} shows that
\begin{align*}
\im(\mathbb{P}(\kappa))\ni x\mapsto \phi(n):=
\mathbbm{1}_{[\kappa,\infty)}(n)\Phi(n,\kappa)x\in \ker(L^+_{\kappa})
\end{align*}
is an isomorphism, which of course implies that $\dim \ker(L_{\kappa}^+)=\dim \im(\mathbb{P}(\kappa))$. This completes the proof.
\end{proof}

\begin{prop}\label{Fredholm-L2}
Assume that the bounded linear discrete vector field $\mathbb{A}\colon \Z\to \mathcal{L}(\R^d)$ admits an ED on $\Z_{\kappa}^-$
with a regular projection $\mathbb{P}\colon \Z_{\kappa}^-\to \mathcal{L}(\R^d)$. Then the operator
$L^-_{\kappa}\colon \ell_{0\kappa}^-(\R^d)\to \ell_{0\kappa-1}^-(\R^d)$ given by
\begin{align*}
(L^-_\kappa\phi)(n):=\mathbbm{1}_{\Z_{\kappa}^-}(n)\cdot((\mathbb{S}_l- \mathbb{A})\phi)(n),
\end{align*}
for all $n\in\Z$ and $\phi\in \ell^-_{0\kappa}(\R^d)$, is surjective and
\begin{align*}
\ker(L^-_\kappa)=\{\phi\in \ell^-_{0\kappa}(\R^d)\mid \phi(n)=\Phi(n,\kappa)x,\;n\leqslant \kappa,\;x\in \ker(\mathbb{P}(\kappa))\}.
\end{align*}
In particular, $\ker(L_{\kappa}^-)$ is isomorphic to $\ker(\mathbb{P}(\kappa))$ and so $\dim \ker(L_{\kappa}^-)=\dim \ker(\mathbb{P}(\kappa))$.
\end{prop}
\begin{proof}
Despite that this proposition looks similar to Proposition \ref{Fredholm-L},
there are some crucial differences which we want to point out.\\
First, repeating the arguments from the proof of Proposition  \ref{Fredholm-L}, one can see that \linebreak
$L^-_{\kappa}\colon \ell_{0\kappa}^{-}(\R^d)\to \ell^{-}_{0\kappa-1}(\R^d)$ is well-defined. In order to show that $L^-_{\kappa}$ is surjective we consider the map
$M\colon \ell^-_{0\kappa-1}(\R^d)\to \ell^-_{0\kappa}(\R^d)$ given by
\begin{align}\label{function-M1}
(M\phi)(n)&=
\sum\limits_{k=-\infty}^{\kappa-1}\mathbbm{1}_{\Z_{\kappa}^-}(n)\mathbb{G}_{\Phi}(n,k+1)\phi(k)=\\
&=\begin{cases}\notag
\sum\limits_{k=-\infty}^{\kappa-1}\Phi(0,k+1)\mathbb{P}(k+1)\phi(k) & \text{ if }\; n=\kappa,\\
\sum\limits_{k=-\infty}^{n-1}\Phi(n,k+1)\mathbb{P}(k+1)\phi(k)-\sum\limits_{k=n}^{\kappa-1}\Phi(n,k+1)(I_d-\mathbb{P}(k+1))\phi(k) & \text{ if }\; n< \kappa,\\
 0 &\text{ if }\; n> \kappa.
\end{cases}
\end{align}
where $\mathbb{G}_{\Phi}$ again is Green's function from \eqref{Green}. Similarly as in the proof of Proposition  \ref{Fredholm-L} one can show that $M\colon \ell^-_{0\kappa-1}(\R^d)\to \ell^-_{0\kappa}(\R^d)$ is well-defined.
To see that $(L^-_{\kappa}M)\phi=\phi$, we note that for $n=\kappa-1$
\begin{align*}
&(L^-_{\kappa}M\phi)(\kappa-1)=\\&
\mathbb{S}_l(M\phi)(\kappa-1)-\mathbb{A}(\kappa-1)(M\phi)(\kappa-1)=(M\phi)(\kappa)-\mathbb{A}(\kappa-1)(M\phi)(\kappa-1)=\\
&\sum\limits_{k=-\infty}^{\kappa-1}\Phi(0,k+1)\mathbb{P}(k+1)\phi(k)-\\
&\left(\sum\limits_{k=-\infty}^{\kappa-2}\mathbb{A}(\kappa-1)\Phi(\kappa-1,k+1)\mathbb{P}(k+1)\phi(k)-
\mathbb{A}(\kappa-1)\Phi(\kappa-1,0)(I_d-\mathbb{P}(0))\phi(\kappa-1)\right)=\\
&\sum\limits_{k=-\infty}^{\kappa-2}\Phi(\kappa,k+1)\mathbb{P}(k+1)\phi(k)+\mathbb{P}(\kappa)\phi(\kappa-1)-
\sum\limits_{k=-\infty}^{-2}\Phi(\kappa,k+1)\mathbb{P}(k+1)\phi(k)+\\&
\Phi(\kappa,\kappa)(I_d-\mathbb{P}(\kappa))\phi(-1)=\mathbb{P}(\kappa)\phi(\kappa-1)+ (I_d-\mathbb{P}(\kappa))\phi(\kappa-1)=\phi(\kappa-1).
\end{align*}
Furthermore, for $n<\kappa-1$ we have

\begin{align}
&(L^-_{\kappa}M\phi)(n)=\mathbb{S}_l(M\phi)(n)-\mathbb{A}(n)(M\phi)(n)=(M\phi)(n+1)-\mathbb{A}(n)(M\phi)(n)=\notag\\
\begin{split}\label{A=Phi}
&\sum^{n}_{k=-\infty} \Phi(n+1,k+1)\mathbb{P}(k+1)\phi(k)-\sum^{\kappa-1}_{k=n+1}\Phi(n+1,k+1)(I_d-\mathbb{P}(k+1))\phi(k)-\\
&\sum^{n-1}_{k=-\infty}\mathbb{A}(n)\Phi(n,k+1)\mathbb{P}(k+1)\phi(k)+\sum^{\kappa-1}_{k=n}\mathbb{A}(n)\Phi(n,k+1)(I_d-\mathbb{P}(k+1))\phi(k).
\end{split}
\end{align}
\noindent As $\mathbb{A}(n)\Phi(n,k+1)=\Phi(n+1,k+1)$, it follows that
\begin{align*}
&\eqref{A=Phi}=\sum^{n}_{k=-\infty} \Phi(n+1,k+1)\mathbb{P}(k+1)\phi(k)-\sum^{\kappa-1}_{k=n+1}\Phi(n+1,k+1)(I_d-\mathbb{P}(k+1))\phi(k)-\\
&\sum^{n-1}_{k=-\infty}\Phi(n+1,k+1)\mathbb{P}(k+1)\phi(k)+\sum^{\kappa-1}_{k=n}\Phi(n+1,k+1)(I_d-\mathbb{P}(k+1))\phi(k)=\\
&\mathbb{P}(n+1)\phi(n)+(I_d-\mathbb{P}(n+1))\phi(n)=\phi(n).
\end{align*}
Finally, we note that $(L^-_{\kappa}M)\phi(n)=0$ for $n>\kappa$, which completes the proof that $L^-_{\kappa}$ is surjective. \newline\indent It remains to describe the kernel of $L^-_{\kappa}\colon \ell^-_{0\kappa}(\R^d)\to \ell^-_{0\kappa-1}(\R^d)$. To this aim, recall that by \eqref{ED3} and Lemma \ref{ED-inverse}
\begin{align*}
|\Phi(n,\kappa)(I_d-\mathbb{P}(\kappa))x|&\leqslant K\alpha^{\kappa-n}|(I_d-\mathbb{P}(\kappa))x|,\\
|||\Phi(\kappa,n)^{-1}(\mathbb{P}(\kappa)x)|||&\geqslant (1/K)(1/\alpha)^{\kappa-n}|\mathbb{P}(\kappa)x|,
\end{align*}
for all $x\in\R^d$ with $n\leqslant \kappa$, $K\geqslant 1$ and $\alpha\in (0,1)$.
If now $\phi\in \ker(L^-_{\kappa})$, then
\begin{align*}
\Phi(\kappa,n)\phi(n)=\phi(\kappa)
\end{align*}
for all $n\leqslant \kappa$. Consequently,

\begin{itemize}
\item if $\phi(\kappa)\in \ker(\mathbb{P}(\kappa))$, then $\phi(n)=\Phi(n,\kappa)\phi(\kappa)$ for all $n\leqslant \kappa$,
\item if $\phi(\kappa)\in \im(\mathbb{P}(\kappa))$, then $\phi(n)\in \Phi(\kappa,n)^{-1}(\phi(\kappa))$ for all $n\leqslant \kappa$.
\end{itemize}
Since $\R^d=\im(\mathbb{P}(\kappa))\oplus \ker(\mathbb{P}(\kappa))$, it suffices to consider the two cases $\phi(\kappa)\in \im(\mathbb{P}(\kappa))$ and $\phi(\kappa)\in \ker(\mathbb{P}(\kappa))$. If $\phi(\kappa)\in \ker(\mathbb{P}(\kappa))$, then
$(I_d-\mathbb{P}(\kappa))\phi(\kappa)=\phi(\kappa)$ and
\begin{align*}
|\phi(n)|=|\Phi(n,\kappa)\phi(\kappa)|=|\Phi(n,\kappa)(I_d-\mathbb{P}(\kappa))\phi(\kappa)|\leqslant K\alpha^{\kappa-n}|(I_d-\mathbb{P}(\kappa))\phi(\kappa)|\xrightarrow[n\rightarrow-\infty]{}0.
\end{align*}
If, on the other hand, $\phi(\kappa)\in \im(\mathbb{P}(\kappa))$, then $\mathbb{P}(\kappa)\phi(\kappa)=\phi(\kappa)$ and $\phi(n)\in \Phi(\kappa,n)^{-1}(\phi(\kappa))$ for all $n\leqslant \kappa$. Consequently, we get by Lemma \ref{ED-inverse}
\begin{equation*}
|\phi(n)|\geqslant |||\Phi(\kappa,n)^{-1}(\mathbb{P}(\kappa)\phi(\kappa))|||\geqslant (1/K)(1/\alpha)^{\kappa-n}|\mathbb{P}(\kappa)\phi(\kappa)|
\xrightarrow[n\rightarrow-\infty]{}\infty,
\end{equation*}
which contradicts the fact that $\phi\in \ell^-_{0\kappa}(\R^d)$. Consequently, we have shown that if $\phi\in \ker(L^-_{\kappa})$, then $\phi(\kappa)\in \ker(\mathbb{P}(\kappa))$. Finally, it follows that
\begin{align*}
\ker(\mathbb{P}(\kappa))\ni x\mapsto \phi(n):=
\mathbbm{1}_{(-\infty,\kappa]}(n)\Phi(n,\kappa)x\in \ker(L^-_{\kappa})
\end{align*}
is an isomorphism, which of course implies that $\dim \ker(L_{\kappa}^-)=\dim \ker(\mathbb{P}(\kappa))$.
This completes the proof.
\end{proof}
\noindent
Now we shall extend the concept of an ED to the case of linear parametrized discrete vector fields $\mathbb{A}\colon \Lambda\times\Z\to \mathcal{L}(\R^d)$,
keeping in mind that $\mathbb{A}$ satisfies $(L0)$.

\begin{defi}\label{ED-Lambda}
A linear parametrized discrete vector field $\mathbb{A}\colon \Lambda\times\Z\to \mathcal{L}(\R^d)$ is said to possess an ED on $\mathbb{I}$
if there exists a parametrized projector $\mathbb{P}\colon \Lambda\times\Z\to \mathcal{L}(\R^d)$, $K\geqslant 1$ and $\alpha\in (0,1)$ such that for all $\lambda\in\Lambda$ one has:
\begin{itemize}
\item $\mathbb{P}(\lambda)\colon \Z\to \mathcal{L}(\R^d)$ is invariant and regular with respect to $\mathbb{A}(\lambda)\colon \Z\to \mathcal{L}(\R^d)$,
\item the pair $(\mathbb{A}(\lambda),\mathbb{P}(\lambda))$ has an ED on $\mathbb{I}$ $($with respect to $K$ and $\alpha)$.
\end{itemize}
\end{defi}
\noindent
The next lemma explains why we assume that the projector $\mathbb{P}\colon \Lambda\times\Z\to \mathcal{L}(\R^d)$
 is continuous with respect to the first variable.

\begin{lemma}\label{lemma-vector-bundles-ker-im}
Assume that $\mathbb{A}\colon \Lambda\times\Z\to \mathcal{L}(\R^d)$ has an ED on $\Z_{\kappa}^{\pm}$ with a projector $\mathbb{P}^{\pm}\colon \Lambda\times\Z_{\kappa}^{\pm}\to \mathcal{L}(\R^d)$. Then
\begin{align}\label{vector-bundels-ker-im}
\begin{split}
\Ker \mathbb{P}^{\pm}(\kappa)&:=\{(\lambda,x)\in \Lambda\times \R^d\mid x\in \ker(\mathbb{P}^{\pm}(\lambda,\kappa)) \},\\
\IM \mathbb{P}^{\pm}(\kappa)&:=\{(\lambda,x)\in \Lambda\times \R^d\mid x\in \im(\mathbb{P}^{\pm}(\lambda,\kappa))\}
\end{split}
\end{align}
are vector bundles over $\Lambda$ such that $\Ker \mathbb{P}^{\pm}(\kappa) \oplus \IM \mathbb{P}^{\pm}(\kappa)=\Theta(\mathbb{R}^d)$.
\end{lemma}

\begin{proof}
Since $\mathbb{P}^{\pm}_{\kappa}\colon \Lambda\to \mathcal{L}(\R^d)$ are continuous and $\Lambda$ is connected, \cite[Chapter 1,~Lemma~4.10]{Kato}
implies that the functions $\Lambda\ni \lambda\mapsto \dim\im(\mathbb{P}^{\pm}(\lambda,\kappa))$
 are constant. Consequently, the assertion follows from \cite[Prop.~14.2.3,~p.~337]{Dieck}.
\end{proof}


It follows from the previous lemma and Corollary \ref{Fredholm-L+} that

\begin{align}\label{stablebundle}
\left\{(\lambda,x)\in \Lambda\times\R^d \;\Big|\;\lim\limits_{n\to\infty}\Phi_{\lambda}(n,\kappa)x=0 \right\}
\end{align}
is a vector bundle over $\Lambda$ which is isomorphic to $\IM \mathbb{P}^{+}(\kappa)$. Moreover, we obtain from Proposition \ref{Fredholm-L2} that

\begin{align}\label{unstablebundle}
\left\{(\lambda,x)\in \Lambda\times\R^d\;|\; x\in \ker \mathbb{P}^-(\lambda,\kappa),\;\lim\limits_{n\to-\infty}\Phi_{\lambda}(n,\kappa)x=0\right\}
\end{align}
is a vector bundle over $\Lambda$ which is isomorphic to $\Ker \mathbb{P}^-(\kappa)$.

\begin{defi}
Let $\mathbb{A}\colon \Lambda\times\Z\to \mathcal{L}(\R^d)$ be as in Definition $\ref{ED-Lambda}$. If $\mathbb{I}=\Z^+_{\kappa}$, then the stable vector bundle $\mathbb{E}^s$ is defined by \eqref{stablebundle}. If $\mathbb{I}=\Z^-_{\kappa}$, then the unstable vector bundle $\mathbb{E}^u$ is defined by \eqref{unstablebundle}.
\end{defi}
\noindent
Now we recall some concepts from Fredholm theory. A bounded operator $T\colon X\rightarrow Y$ acting between Banach spaces $X$, $Y$ is called \textit{Fredholm} if it has finite dimensional kernel and cokernel. The \textit{Fredholm index} of $T$ is the integer
\begin{align}\label{Fred-index}
\ind(T)=\dim\ker(T)-\dim\coker(T).
\end{align}
Note that if $\dim X<\infty$ and $\dim Y<\infty$, then
\begin{align}\label{Fred-index-finite}
\ind(T)=\dim X-\dim Y.
\end{align}
The subspace of $\mathcal{L}(X,Y)$ consisting of all Fredholm operators will be denoted by $\Phi(X,Y)$. Moreover, $\Phi_k(X,Y)$, $k\in\mathbb{Z}$, stands for the subset of all operators in $\Phi(X,Y)$ having index $k$.\\
In this paper we need the concept of the \textit{index bundle} for families $L\colon \Lambda\rightarrow\Phi(X,Y)$ of Fredholm operators parametrized by a compact topological space $\Lambda$ that generalizes the integral Fredholm index \eqref{Fred-index} to families. Given a family $L\colon \Lambda\to \Phi(X,Y)$, we can define a map $\mathbf{L}\colon \Lambda\times X\to Y$ by $\mathbf{L}(\lambda,x)=L(\lambda)x$. By a composition of two families $L\colon \Lambda\to \Phi(X,Y)$ and $S\colon \Lambda\to \Phi(Y,Z)$, denoted by $S\diamond L$, we shall mean the family $S\diamond L\colon \Lambda\to \Phi(X,Z)$ defined by
$(S\diamond L)(\lambda)x:= (\mathbf{S}\circ (\pr_{\Lambda},\mathbf{L}))(\lambda,x)$,
where $\mathbf{S}$ and $\mathbf{L}$ are induced as above by $S$ and $L$, respectively, and $\pr_{\Lambda}\colon \Lambda\times X\to\Lambda$ is the projection.\\
Let us now recall the construction of the \textit{index bundle} for families $L\colon \Lambda\rightarrow\Phi(X,Y)$ of Fredholm operators (cf. e.g. \cite{KTheoryAtiyah}, \cite{indbundleIch}). By the compactness of $\Lambda$, there is a finite dimensional subspace $V\subset Y$ such that
\begin{align}\label{subspace}
\im(L_\lambda)+V=Y,\quad\lambda\in\Lambda,
\end{align}
i.e. $V$ is transversal to the images of $L_{\lambda}$. As $V$ is finite dimensional, there is a projection $P$ onto $V$, and we obtain a family of exact sequences
\begin{align*}
X\xrightarrow{L_\lambda}Y\xrightarrow{I_Y-P}\im(I_Y-P).
\end{align*}
This yields a vector bundle $E(L,V)$ consisting of the union of the kernels of the maps $(I_Y-P)\circ L_\lambda$, $\lambda\in\Lambda$ (cf. \cite[\S III.3]{Lang}), and the fibres of $E(L,V)$ are given by $L_{\lambda}^{-1}(V)$. Now, if $\Theta(V)$ stands for the product bundle $\Lambda\times V$, then we obtain a $KO$-theory class
\begin{align}\label{defiindbund}
\ind(L):=[E(L,V)]-[\Theta(V)]\in KO(\Lambda),
\end{align}
which is called the \textit{index bundle} of $L$. It is easy to check that if $\Lambda$ is connected, then
\begin{align*}
\dim E(L,V)=\dim(V)+\ind(L_\lambda),\quad\lambda\in\Lambda.
\end{align*}
Consequently, we see that
\begin{itemize}
\item
$\ind(L)\in\widetilde{KO}(\Lambda)$ if and only if the operators $L_\lambda$ are Fredholm of index $0$,
\item if $\Lambda=\{\lambda_0\}$ is a singleton, then

\[\ind(L)=\dim(E(L,V))-\dim(V)=\ind(L_{\lambda_0})\in KO(\Lambda)\simeq\mathbb{Z},\]
which implies that the definitions \eqref{defiindbund} and \eqref{Fred-index} coincide in this case.
\end{itemize}
Let us recall the following properties of the index bundle (cf. \cite{indbundleIch}):
\begin{enumerate}
\item[(1)] (Normalization) If $L_\lambda$ is invertible for all $\lambda\in\Lambda$, then $\ind(L)=0$.
\item[(2)] (Functoriality) If $f\colon \Lambda'\rightarrow\Lambda$ is a continuous map between compact spaces and $L\colon\Lambda\to \Phi(X,Y)$ is a family of Fredholm operators, then $f^\ast L\colon\Lambda'\rightarrow\Phi(X,Y)$ defined by $(f^\ast L)(\lambda)=L(f(\lambda))$, $\lambda\in \Lambda'$, is a family of Fredholm operators and $\ind(f^\ast L)=f^\ast(\ind(L))$, where $f^{\ast}\colon KO(\Lambda)\to KO(\Lambda')$ is the group homomorphism induced by $f$.
\item[(3)] (Homotopy invariance) If $H:[0,1]\times\Lambda\rightarrow\Phi(X,Y)$ is a homotopy of Fredholm operators, then $$\ind(H(0,\cdot))=\ind(H(1,\cdot))\in KO(\Lambda).$$
\item[(4)] (Compact perturbation) If $K\colon\Lambda\rightarrow\mathcal{L}(X,Y)$ is a family of compact operators, then
$$\ind(L)=\ind(L+K)\in KO(\Lambda).$$
\item[(5)] (Logarithmic property) If $S\colon\Lambda\rightarrow\Phi(Y,Z)$ and $L\colon\Lambda\rightarrow\Phi(X,Y)$ are two families of Fredholm operators, then $$\ind(S\diamond L)=\ind(S)+\ind(L)\in KO(\Lambda).$$
\item[(6)] (Additivity) If $L\colon\Lambda\rightarrow\Phi(X,Y)$ and $\widetilde{L}\colon\Lambda\rightarrow\Phi(\widetilde{X},\widetilde{Y})$ are two families of Fredholm operators, then
$$\ind(L\oplus \widetilde{L})=\ind(L)+\ind(\widetilde{L})\in KO(\Lambda)\in KO(\Lambda).$$
\item[(7)] (Finite-dimensional property) If $L\colon \Lambda\to \mathcal{L}(X,Y)$ with $\dim(X)=p$ and $\dim(Y)=q$, then

\begin{align}\label{finite-dim}
\ind(L)=[\Theta(\mathbb{R}^p)]-[\Theta(\mathbb{R}^q)].
\end{align}

\end{enumerate}
\noindent
Let us emphasize that the following result is well-known (comp. \cite{Christian10}). However, we present a new proof, which is essential for the argument of Theorem \ref{ind-bundle-formula} below.

\begin{lemma}\label{fredholm}
If $\mathbb{A}\colon \Z\to \mathcal{L}(\R^d)$ admits an ED both on $\Z^+_{\overline{\kappa}}$ and on $\Z^-_{\underline{\kappa}}$ with respective $($regular$)$ projectors $\mathbb{P}^+\colon \Z^+_{\overline{\kappa}}\to \mathcal{L}(\R^d)$ and
$\mathbb{P}^-\colon\Z^-_{\underline{\kappa}}\to \mathcal{L}(\R^d)$, for some $\underline{\kappa}< 0<\overline{\kappa}$, then the linear operator $L:=\mathbb{S}_l-\mathcal{N}_{\mathbb{A}}\colon \ell_0(\R^d)\to \ell_0(\R^d)$ is Fredholm and
\begin{equation*}
\ind(L)=\dim \im(\mathbb{P}^+(\overline{\kappa}))-\dim\im (\mathbb{P}^-(\underline{\kappa}))\in \Z.
\end{equation*}
\end{lemma}
\begin{proof}
We first note that
\begin{align*}
\ell_0(\R^d)=\ell_{0\underline{\kappa}-1}^-(\R^d)\oplus \ell_{\underline{\kappa}\overline{\kappa}-1}(\R^d)\oplus \ell_{0\overline{\kappa}}^+(\R^d)
\end{align*}
and consider the commutative diagram

\begin{align}\label{diagram-3-L}
\begin{split}
\xymatrix@C=65pt@R=20pt{ \ell_0(\R^d)\ar[r]^-{L} \ar[d]_{J}& \ell_0(\R^d)
\\
 \ell_{0\underline{\kappa}}^-(\R^d)\oplus \ell_{\underline{\kappa}\overline{\kappa}}(\R^d)\oplus \ell_{0\overline{\kappa}}^+(\R^d)
 \ar[r]^-{L^-_{\underline{\kappa}}\oplus L^{c}\oplus L^+_{\overline{\kappa}} } & \ar[u]_{I}
  \ell_{0\underline{\kappa}-1}^-(\R^d)\oplus \ell_{\underline{\kappa}\overline{\kappa}-1}(\R^d)\oplus \ell_{0\overline{\kappa}}^+(\R^d),
 }
 \end{split}
\end{align}
where $I(\phi,\psi,\varphi)=\phi+\psi+\varphi$ and $J(\phi)=(\mathbbm{1}_{\Z\cap(-\infty,\underline{\kappa}]}\phi,\mathbbm{1}_{\Z\cap [\underline{\kappa},\overline{\kappa}]}\phi,\mathbbm{1}_{\Z\cap [\overline{\kappa},\infty)}\phi)$, i.e.,
\begin{align*}
J\phi=
(...\phi(\underline{\kappa}-1),\phi(\underline{\kappa}),\phi(\underline{\kappa}),\phi(\underline{\kappa}+1),\ldots,\phi(\overline{\kappa}),\phi(\overline{\kappa}),\phi(\overline{\kappa}+1),...)
\end{align*}
with
\begin{align*}
& L^{-}_{\underline{\kappa}}\colon  \ell_{0\underline{\kappa}}^-(\R^d)\to \ell_{0\underline{\kappa}-1}^-(\R^d),& \;L^{-}_{\underline{\kappa}}\phi&=\mathbbm{1}_{\Z\cap(-\infty,\underline{\kappa}-1]}(\mathbb{S}_l- \mathcal{N}_{\mathbb{A}})\phi,\\
& L^{c}\colon  \ell_{\underline{\kappa}\overline{\kappa}}(\R^d)\to \ell_{\underline{\kappa}\overline{\kappa}-1}(\R^d),&  L^{c}\phi&=\mathbbm{1}_{\Z\cap [\underline{\kappa},\overline{\kappa}]}(\mathbb{S}_l-\mathcal{N}_{\mathbb{A}})\phi,\\
& L^{+}_{\overline{\kappa}}\colon  \ell_{0\overline{\kappa}}^+(\R^d)\to \ell_{0\overline{\kappa}}^+(\R^d),& L^{+}_{\overline{\kappa}}\phi&=\mathbbm{1}_{\Z\cap[\overline{\kappa},+\infty)}(\mathbb{S}_l- \mathcal{N}_{\mathbb{A}})\phi.
\end{align*}
Clearly, $I$ is an isomorphism, and hence $I$ is Fredholm with $\ind(I)=0$. Moreover, $J$ is injective with
$\IM(J)=\Ker(\mathcal{R})$,
where $\mathcal{R}\colon  \ell_{0\underline{\kappa}}^-(\R^d)\oplus \ell_{\underline{\kappa}\overline{\kappa}}(\R^d)\oplus \ell_{0\overline{\kappa}}^+(\R^d)
\to \R^d\times \R^d$ is given by
\begin{align}\label{the-map-R}
\mathcal{R}(\phi,\psi,\varphi)=(\phi(\underline{\kappa})-\psi(\underline{\kappa}),\psi(\overline{\kappa})-\varphi(\overline{\kappa})),
\end{align}
for $\phi\in \ell_{0\underline{\kappa}}^-(\R^d)$, $\psi\in \ell_{\underline{\kappa}\overline{\kappa}}(\R^d)$ and $\varphi\in \ell_{0\overline{\kappa}}^+(\R^d)$.
Since $\IM(\mathcal{R})=\R^d\times\R^d$, it follows that
\begin{align*}
\ell_{0\underline{\kappa}}^-(\R^d)\oplus \ell_{\underline{\kappa}\overline{\kappa}}(\R^d)\oplus \ell_{0\overline{\kappa}}^+(\R^d)/\Ker(\mathcal{R})\simeq \R^d\times \R^d.
\end{align*}
Hence there exists a finite dimensional subspace $V\subset \ell_{0\underline{\kappa}}^-(\R^d)\oplus \ell_{\underline{\kappa}\overline{\kappa}}(\R^d)\oplus \ell_{0\overline{\kappa}}^+(\R^d)$ such that
\begin{align}
\ell_{0\underline{\kappa}}^-(\R^d)\oplus \ell_{\underline{\kappa}\overline{\kappa}}(\R^d)\oplus \ell_{0\overline{\kappa}}^+(\R^d)=\IM(J)\oplus V \text{ and }\dim V=2d,
\end{align}
which implies that $\Coker(J)\simeq \R^{2d}$. Consequently, $J$ is a Fredholm operator with $\ind(J)=-2d$.
Moreover, Corollary \ref{Fredholm-L+} and Proposition \ref{Fredholm-L2} imply that $L^+_{\overline{\kappa}}$ and $L^-_{\underline{\kappa}}$ are Fredholm maps with
\begin{align*}
\ind(L^+_{\overline{\kappa}})&=\dim \ker(L^+_{\overline{\kappa}})-\dim\coker(L^+_{\overline{\kappa}})=\dim \ker(L^+_{\overline{\kappa}})=\dim \im(\mathbb{P}(\overline{\kappa})),\\
\ind(L^-_{\underline{\kappa}})&=\dim \ker(L^-_{\underline{\kappa}})-\dim\coker(L^-_{\underline{\kappa}})=\dim \ker(L^-_{\underline{\kappa}})=\dim \ker(\mathbb{P}(\underline{\kappa})).
\end{align*}
Finally, for $L^c \colon  \ell_{\underline{\kappa}\overline{\kappa}}(\R^d)\to \ell_{\underline{\kappa}\overline{\kappa}-1}(\R^d)$, note that
\begin{align}\label{dim-spaces-L}
\ell_{\underline{\kappa}\overline{\kappa}}(\R^d)\simeq \prod\limits_{i=1}^{\overline{\kappa}-\underline{\kappa}+1}\R^d \text{ and }
\ell_{\underline{\kappa}\overline{\kappa}-1}(\R^d)\simeq \prod\limits_{i=1}^{\overline{\kappa}-\underline{\kappa}}\R^d.
\end{align}
Consequently, taking into account \eqref{Fred-index-finite} and \eqref{dim-spaces-L}, we infer that
\begin{align*}
\ind(L^c)=\dim( \ell_{\underline{\kappa}\overline{\kappa}}(\R^d))-\dim(\ell_{\underline{\kappa}\overline{\kappa}-1}(\R^d))=
(\overline{\kappa}-\underline{\kappa}+1)d-(\overline{\kappa}-\underline{\kappa})d=d.
\end{align*}
As the composition of Fredholm operators is Fredholm and
$L=I\circ (L^-_{\underline{\kappa}}\oplus L^c\oplus L^+_{\overline{\kappa}})\circ J$, it follows that $L$ is Fredholm and
\begin{align*}
\ind(L)&=\ind(I\circ (L^-_{\underline{\kappa}}\oplus L^c\oplus L^+_{\overline{\kappa}})\circ J)
=\ind(I)+\ind(L^-_{\underline{\kappa}}\oplus L^c\oplus L^+_{\overline{\kappa}})+\ind(J)\\
&=0+\ind(L^-_{\underline{\kappa}}\oplus L^c\oplus L^+_{\overline{\kappa}})-2d
=\ind(L^-_{\underline{\kappa}})+ \ind(L^c)+ \ind(L^+_{\overline{\kappa}})-2d\\
&=\dim \ker(\mathbb{P}^-(\underline{\kappa}))+d+ \dim \im(\mathbb{P}^+(\overline{\kappa}))-2d
=\dim \im(\mathbb{P}^+(\overline{\kappa}))-(d-\dim \ker(\mathbb{P}^-(\underline{\kappa})))\\
&=\dim \im(\mathbb{P}^+(\overline{\kappa}))-\dim \im(\mathbb{P}^-(\underline{\kappa})),
\end{align*}
where we have used the rank-nullity theorem in the final equality. This completes the proof.
\end{proof}

\begin{lemma}[\cite{Poetzscheb}]\label{iso}
$1\not\in \Sigma(\mathbb{A})\Longleftrightarrow L_{\mathbb{A}}\colon \ell_0(\R^d)\to \ell_0(\R^d)$ is an isomorphism.
\end{lemma}

Now we carry over Lemma \ref{fredholm} to the case of parametrized linear vector fields.

\begin{theorem}\label{ind-bundle-formula}
Assume that $\mathbb{A}\colon \Lambda\times\Z\rightarrow\mathcal{L}(\R^d)$ admits an ED
both on $\Z^+_{\overline{\kappa}}$ and on $\Z^-_{\underline{\kappa}}$  with projectors $\mathbb{P}^+\colon \Lambda\times\Z_{\overline{\kappa}}^+\to \mathcal{L}(\R^d)$ and $\mathbb{P}^-\colon \Lambda\times\Z^-_{\underline{\kappa}}\to \mathcal{L}(\R^d)$, where $\underline{\kappa}<0<\overline{\kappa}$.
Then the map $L_\mathbb{A}\colon \Lambda\to \mathcal{L}(\ell_0(\R^d))$ defined by $L_\mathbb{A}(\lambda):=\mathbb{S}_l-\mathcal{N}_{\mathbb{A}(\lambda)}$, $\lambda\in\Lambda$,
has the following properties:
\begin{enumerate}
\item[$(a)$] $L_\mathbb{A}$ is a continuous family of Fredholm operators,
\item[$(b)$]  the index bundle of $L_\mathbb{A}$ is
\begin{equation}\label{index-KO}
\ind(L_\mathbb{A})=[\IM \mathbb{P}^+(\overline{\kappa})]-[\IM \mathbb{P}^-(\underline{\kappa})]\in KO(\Lambda),
\end{equation}
where the bundles on the right hand side are defined in \eqref{vector-bundels-ker-im}.
\end{enumerate}

\end{theorem}
\begin{proof}
Note that $(a)$ follows immediately from Lemma \ref{L-Niemytski} and Lemma \ref{fredholm}.
To prove $(b)$ we consider the parametrized version of the diagram \eqref{diagram-3-L}:
\begin{align*}
\xymatrix@C=55pt@R=20pt{ \Lambda\times\ell_0(\R^d)\ar[r]^-{\id} \ar[d]_{\mathbf{J}}& \Lambda\times\ell_0(\R^d) \ar[ddd]^{\mathbf{L}}
\\
 \big(\Lambda\times\ell_{0\underline{\kappa}}^{-}(\R^d)\big)\oplus \big(\Lambda\times\ell_{\underline{\kappa}\overline{\kappa}}(\R^d)\big)\oplus
 (\Lambda\times\ell_{0\overline{\kappa}}^+(\R^d))
 \ar[d]_-{\mathbf{L}^-_{\underline{\kappa}}\oplus \mathbf{L}^{\!c}\oplus \mathbf{L}^+_{\overline{\kappa}} } &  \\
  \ell_{0\underline{\kappa}-1}^{-}(\R^d)\oplus \ell_{\underline{\kappa}\overline{\kappa}-1}(\R^d)\oplus \ell_{0\overline{\kappa}}^+(\R^d)\ar[d]_{I} & \\
  \ell_0(\R^d) & \ar[l]_{\id}\ell_0(\R^d),
 }
\end{align*}
where $I(\phi,\psi,\varphi)=\phi+\psi+\varphi$, $\mathbf{J}(\lambda,\phi)=((\lambda,J^-(\phi)),(\lambda,J^c(\phi)),(\lambda,J^+(\phi)))$ with
$J^-(\phi)=\mathbbm{1}_{\Z\cap (-\infty,\underline{\kappa}]}\phi$, $J^c(\phi)=\mathbbm{1}_{\Z\cap [\underline{\kappa},\overline{\kappa}]}\phi$,
$J^+(\phi)=\mathbbm{1}_{\Z\cap [\overline{\kappa},\infty)}\phi$
and

\begin{align*}
& \mathbf{L}\colon \Lambda\times\ell_0(\R^d)\to \ell_0(\R^d)& &\mathbf{L}(\lambda,\phi)=L(\lambda)\phi,\\
& \mathbf{L}^{-}_{\underline{\kappa}}\colon  \Lambda\times\ell_{0\underline{\kappa}}^-(\R^d)\to \ell_{0\underline{\kappa}-1}^-(\R^d),& \;&\mathbf{L}^{-}_{\underline{\kappa}}(\lambda,\phi)=\mathbbm{1}_{\Z\cap (-\infty,\underline{\kappa}-1]} (\mathbb{S}_l\phi-\mathcal{N}_{\mathbb{A}}(\lambda,\phi)),\\
& \mathbf{L}^{c}\colon \Lambda\times \ell_{\underline{\kappa}\overline{\kappa}}(\R^d)\to \ell_{\underline{\kappa}\overline{\kappa}-1}(\R^d),&  &\mathbf{L}^{c}(\lambda,\phi)=\mathbbm{1}_{\Z\cap [\underline{\kappa},\overline{\kappa}-1]}(\mathbb{S}_l\phi- \mathcal{N}_{\mathbb{A}}(\lambda,\phi)),\\
& \mathbf{L}^{+}_{\overline{\kappa}}\colon  \Lambda\times\ell_{0\overline{\kappa}}^+(\R^d)\to \ell_{0\overline{\kappa}}^+(\R^d),& &\mathbf{L}^{+}_{\overline{\kappa}}(\lambda,\phi)=\mathbbm{1}_{\Z\cap[\overline{\kappa},\infty)}(\mathbb{S}_l\phi- \mathcal{N}_{\mathbb{A}}(\lambda,\phi)).
\end{align*}
Note that $I$ is an isomorphism, while $\mathbf{J}_{\lambda}$, $\lambda\in\Lambda$, is injective with
$\IM(\mathbf{J}_{\lambda})=\Ker(\mathcal{R})$, where $\mathcal{R}$ is given as in \eqref{the-map-R}. From the proof of Lemma \ref{fredholm} it follows that there exists a finite-dimensional subspace $V\subset \ell_{0\underline{\kappa}}^-(\R^d)\oplus \ell_{\underline{\kappa}\overline{\kappa}}(\R^d)\oplus \ell_{0\overline{\kappa}}^+(\R^d)$ such that for all $\lambda\in\Lambda$ it holds
\begin{align}\label{the-map-J}
\IM(\mathbf{J}_{\lambda})\oplus V=\ell_{0\underline{\kappa}}^-(\R^d)\oplus \ell_{\underline{\kappa}\overline{\kappa}}(\R^d)\oplus \ell_{0\overline{\kappa}}^+(\R^d).
\end{align}
The commutativity of the above diagram implies that
$\mathcal{I}\diamond (L^-_{\underline{\kappa}}\oplus L^{\!c}\oplus L^+_{\overline{\kappa}})\diamond \mathcal{J}=L$,
where
\begin{align*}
&\mathcal{J}\colon\Lambda\to \mathcal{L}(\ell_0(\R^d),\ell_{0\underline{\kappa}}^{-}(\R^d)\oplus \ell_{\underline{\kappa}\overline{\kappa}}(\R^d)\oplus\ell_{0\overline{\kappa}}^+(\R^d)),
&&\mathcal{J}(\lambda)\phi=\mathbf{J}(\lambda,\phi),\\
&\mathcal{I}\colon\Lambda\to \mathcal{L}(\ell_{0\underline{\kappa}-1}^{-}(\R^d)\oplus \ell_{\underline{\kappa}\overline{\kappa}-1}(\R^d)\oplus \ell_{0\overline{\kappa}}^+(\R^d),\ell_0(\R^d)),
 &&\mathcal{I}(\lambda)(\phi,\psi,\varphi)=\phi+\psi+\varphi,\\
& L^{-}_{\underline{\kappa}}\colon  \Lambda\to \mathcal{L}(\ell_{0\underline{\kappa}}^-(\R^d),\ell_{0\underline{\kappa}-1}^-(\R^d)),
&& L^{-}_{\underline{\kappa}}(\lambda)(\phi)=\mathbf{L}^{-}_{\underline{\kappa}}(\lambda,\phi),\\
& L^{c}\colon \Lambda\to \mathcal{L}(\ell_{\underline{\kappa}\overline{\kappa}}(\R^d),\ell_{\underline{\kappa}\overline{\kappa}-1}(\R^d)),
&&  L^{c}(\lambda)(\phi)= \mathbf{L}^{c}(\lambda,\phi),\\
& L^{+}_{\overline{\kappa}}\colon  \Lambda\to \mathcal{L}(\ell_{0\overline{\kappa}}^+(\R^d),\ell_{0\overline{\kappa}}^+(\R^d)),
&&  L^{+}_{\overline{\kappa}}(\lambda)(\phi)= \mathbf{L}^{+}_{\overline{\kappa}}(\lambda,\phi).
\end{align*}
Since $\mathcal{I}$ is a family of isomorphisms, the normalization property of the index bundle implies that
$\ind(\mathcal{I})=0$. Moreover, \eqref{the-map-J} implies that $E(\mathcal{J},V)=\Theta(\{0\})$ and hence
\begin{align*}
\ind(\mathcal{J})=[E(\mathcal{J},V)]-[\Theta(V)]=[\Theta(\{0\})]-[\Theta(\R^{2d})].
\end{align*}
Furthermore, Corollary \ref{Fredholm-L+} and Proposition \ref{Fredholm-L2} show that
$L^{-}_{\underline{\kappa}}$ and $L^{+}_{\overline{\kappa}}$ are families of epimorphism and hence
we see that the trivial subspace $\{0\}$ is transversal to the images of $L^{-}_{\underline{\kappa}}$ and $L^{+}_{\overline{\kappa}}$ as in \eqref{subspace}.
Consequently, we have
\begin{align*}
\ind(L^{-}_{\underline{\kappa}})=[E(L^{-}_{\underline{\kappa}},\{0\})]-[\Theta(\{0\})],\quad
\ind(L^{+}_{\overline{\kappa}})=[E(L^{+}_{\overline{\kappa}},\{0\})]-[\Theta(\{0\})].
\end{align*}
In addition, as the fibres of $E(L^{-}_{\underline{\kappa}},\{0\})$ and $E(L^{+}_{\overline{\kappa}},\{0\})$ are the kernels of $L^{-}_{\underline{\kappa}}$
and $L^{+}_{\overline{\kappa}}$, respectively, Corollary \ref{Fredholm-L+} and Proposition \ref{Fredholm-L2} imply that the vector bundles $E(L^{-}_{\underline{\kappa}},\{0\})$ and $E(L^{+}_{\overline{\kappa}},\{0\})$ are isomorphic to  $\Ker \mathbb{P}^{-}(\underline{\kappa})$ and  $\IM \mathbb{P}^{+}(\overline{\kappa})$ given by \eqref{vector-bundels-ker-im}, respectively. Thus we obtain
\begin{align}
\ind(L^{-}_{\underline{\kappa}})=[\Ker \mathbb{P}^{-}(\underline{\kappa})]-[\Theta(\{0\})],\;\;
\ind(L^{+}_{\overline{\kappa}})=[\IM \mathbb{P}^{+}(\overline{\kappa})]-[\Theta(\{0\})],\label{index-L+}
\end{align}
while the index bundle of $L^c \colon \Lambda\to \mathcal{L}(\ell_{\underline{\kappa}\overline{\kappa}}(\R^d),\ell_{\underline{\kappa}\overline{\kappa}-1}(\R^d))$, by the finite-dimensional property formulated in \eqref{finite-dim} and \eqref{dim-spaces-L}, is equal to
\begin{align*}
\ind(L^{c})=[\Theta(\R^{(\overline{\kappa}-\underline{\kappa}+1)d)}]-[\Theta(\R^{(\overline{\kappa}-\underline{\kappa})d)}]
=[\Theta(\R^d)]-[\Theta(\{0\})].
\end{align*}
Finally, taking all this into account, we obtain
\begin{align*}
\ind(L_\mathbb{A})&=\ind(\mathcal{I}\circ (L^-_{\underline{\kappa}}\oplus L^c\oplus L^+_{\overline{\kappa}})\circ \mathcal{J})
=\ind(\mathcal{I})+\ind(L^-_{\underline{\kappa}}\oplus L^c\oplus L^+_{\overline{\kappa}})+\ind(\mathcal{J})\\
&
=0+\ind(L^-_{\underline{\kappa}})+ \ind(L^c)+ \ind(L^+_{\overline{\kappa}})+\ind(\mathcal{J})\\
&=\ind(L^-_{\underline{\kappa}})+ \ind(L^+_{\overline{\kappa}})+\ind(L^c)+\ind(\mathcal{J})
\\
&=\left([\IM \mathbb{P}^{+}(\overline{\kappa})]-[\Theta(\{0\})]\right)
+\left([\Ker \mathbb{P}^{-}(\underline{\kappa})]-[\Theta(\{0\})]\right)+\\
&\quad\quad\quad\quad\quad\quad+\left([\Theta(\R^d)]-[\Theta(\{0\})]\right)+\left([\Theta(\{0\})]-[\Theta(\R^{2d})]\right)\\
&=\left([\IM \mathbb{P}^{+}(\overline{\kappa})]-[\Theta(\{0\})]\right)
+\left([\Ker \mathbb{P}^{-}(\underline{\kappa})]-[\Theta(\{0\})]\right)+\left([\Theta(\{0\})]-[\Theta(\R^d)]\right)\\
&=\left([\IM \mathbb{P}^{+}(\overline{\kappa})]-[\Theta(\{0\})]\right)+\left([\Ker \mathbb{P}^{-}(\underline{\kappa})]-[\Theta(\R^d)]\right)\\
&=\left([\IM \mathbb{P}^{+}(\overline{\kappa})]-[\Theta(\{0\})]\right)+\left(
[\Ker \mathbb{P}^{-}(\underline{\kappa})]-[\Ker \mathbb{P}^{-}(\underline{\kappa})]+[\Theta(\{0\})]-[\IM \mathbb{P}^{-}(\underline{\kappa})]\right)
\\
&=\left([\IM \mathbb{P}^{+}(\overline{\kappa})]-[\Theta(\{0\})]\right)+\left([\Theta(\{0\})]-[\IM \mathbb{P}^-(\underline{\kappa})]\right)\\
&=[\IM \mathbb{P}^{+}(\overline{\kappa})]-[\IM \mathbb{P}^{-}(\underline{\kappa})],
\end{align*}
as claimed in (b).
\end{proof}

\begin{rem}\label{rem-iso}
It follows from Lemma \ref{iso} that the family $\mathbb{A}$ in Theorem \ref{ind-bundle-formula} satisfies the assumption $(F3)$ if and only if $L_{\mathbb{A}(\lambda_0)}$ is an isomorphism. As isomorphisms are Fredholm operators of index $0$, and $\Lambda$ is connected, $L_\mathbb{A}$ is a family of Fredholm operators of index $0$ in this case.
\end{rem}


\noindent
We conclude this section by two corollaries of the previous theorem and Theorem \ref{perturbed-ED}.

\begin{cor}\label{ind-perturbation}
Let $\mathbb{A}:\Lambda\times\mathbb{Z}\rightarrow\mathcal{L}(\mathbb{R}^d)$ be asymptotically hyperbolic and $\mathbb{D}:\Lambda\times\mathbb{Z}\rightarrow\mathcal{L}(\mathbb{R}^d)$ $\gamma_\pm$-small at $\pm\infty$, where $\gamma_\pm$ is as in Theorem \ref{perturbed-ED}. Then $L_{\mathbb{A}+\mathbb{D}}:\Lambda\rightarrow\mathcal{L}(\ell_0(\mathbb{R}^d))$ is a continuous family of Fredholm operators and

\[\ind(L_{\mathbb{A}+\mathbb{D}})=\ind(L_\mathbb{A})\in KO(\Lambda),\]
where $L_\mathbb{A}$ and $L_{\mathbb{A}+\mathbb{D}}$ are defined as in Theorem \ref{ind-bundle-formula}.
\end{cor}

\begin{proof}
We consider the homotopy $H:[0,1]\times\Lambda\rightarrow\mathcal{L}(\ell_0(\mathbb{R}^d))$ given by $H(t,\cdot)=L_{\mathbb{A}+t\mathbb{D}}$. As $H$ is continuous by Lemma \ref{N-L-continuous} and each $H(t,\lambda)$ is a Fredholm operator by Theorem \ref{perturbed-ED} and Lemma \ref{fredholm}, the assertion follows from the homotopy invariance property of the index bundle.
\end{proof}

The following corollary is a generalization of the main theorem of our previous work \cite{SkibaIch}, where we considered asymptotically hyperbolic systems having invertible limits. Let us recall from Section \ref{versus} that $P^s_{\mathcal{A}}$ denotes the spectral projection of a hyperbolic matrix $\mathcal{A}$ with respect to eigenvalues inside the unit circle. For families of hyperbolic matrices, the images of the projections are vector bundles over the parameter space, which follows by the same argument as Lemma \ref{lemma-vector-bundles-ker-im}.

\begin{cor}\label{cor:asympFred}
If $\mathbb{A}:\Lambda\times\mathbb{Z}\rightarrow\mathcal{L}(\mathbb{R}^d)$ is asymptotically hyperbolic with uniform limits $\mathcal{A}_{\pm}:\Lambda\rightarrow H(d,\R)$, then

\[\ind(L_{\mathbb{A}})=[\IM P^s_{\mathcal{A}_+}]-[\IM P^s_{\mathcal{A}_-}]\in KO(\Lambda).\]
\end{cor}

\begin{proof}
We define $\mathcal{A}\colon \Lambda\times \Z\to \mathcal{L}(\R^d)$ by
\begin{equation*}
\mathcal{A}(\lambda,n)=\left\{
\begin{array}{ll}
\mathcal{A}_+(\lambda) & \emph{ if }\; n\geqslant 0, \\
\mathcal{A}_-(\lambda) & \emph{ if }\; n<0,
\end{array}
\right.
\end{equation*}
and note that
\begin{align*}
\mathbb{A}=\mathcal{A}+(\mathbb{A}-\mathcal{A}) \text{ as well as }
\lim_{n\to\pm\infty} (\mathbb{A}(\lambda,n)-\mathcal{A}(\lambda,n))=0
\end{align*}
uniformly in $\lambda\in\Lambda$. By the latter observation, $\mathbb{A}-\mathcal{A}$ is $\gamma_\pm$-small at $\pm\infty$ for any $\gamma_\pm>0$. Hence we obtain from Theorem \ref{ind-bundle-formula} and Corollary \ref{ind-perturbation} that
\begin{align*}
&\ind(L_{\mathbb{A}})=\ind(L_{\mathcal{A}+(\mathbb{A}-\mathcal{A})})=\ind(L_{\mathcal{A}})=[\IM \mathbb{P}^+_{\mathcal{A}}(\overline{\kappa})]-[\IM \mathbb{P}^-_{\mathcal{A}}(\underline{\kappa})]\\
&=[\IM P^s_{\mathcal{A}_+}]-[\IM P^s_{\mathcal{A}_-}],
\end{align*}
where $\overline{\kappa}\in \Z$ is sufficiently large, $\underline{\kappa}\in \Z$ is sufficiently small, and the last equality follows from Proposition \ref{prop-projconstant}.
\end{proof}


\section{Bifurcation results for discrete dynamical systems with examples}\label{section-bifurcation}
The aim of this section is to weaken the assumptions of the bifurcation results from  \cite{JacoboRobertII,SkibaIch} for which we use the material developed in the previous sections. As most of the results follow from Theorem \ref{ind-bundle-formula} and Lemma \ref{L-Niemytski} with similar arguments as in \cite{SkibaIch}, we will just give precise references instead of full arguments. In the second part of this section we give examples to show how the following theorems generalise our previous results for asymptotically hyperbolic systems from \cite{SkibaIch}.

\begin{theorem}\label{thm:nonlin}
If the system \eqref{nonlinear-equation} satisfies the assumptions \emph{(F0)--(F3)} and
\begin{align}\label{J-condition}
J(\IM \mathbb{P}^+(\overline{\kappa}))\neq J(\IM \mathbb{P}^-(\underline{\kappa}))\in J(\Lambda),
\end{align}
then there is a bifurcation point.
\end{theorem}

\begin{proof}
The argument follows the proof of Theorem 2.2 in \cite{SkibaIch}, which can be found in \S 4 of that reference. Indeed, the analytic properties of the nonlinearity in \eqref{nonlinear-equation} that were before obtained in \cite[\S 4.1]{SkibaIch} now follow from the more general Lemma \ref{L-Niemytski}. The second part of the proof in \cite[\S 4.2]{SkibaIch} carries over verbatim. Note that we are indeed dealing with Fredholm operators of index $0$ by Remark \ref{rem-iso}. Finally, in the third part \cite[\S 4.3]{SkibaIch} we just have to replace the index formula from \cite[Thm. 3.2]{SkibaIch} by our new Theorem \ref{ind-bundle-formula}.
\end{proof}
\noindent
It is well known that the non-triviality of the J-homomorphism can be obtained from Stiefel-Whitney classes, i.e., if $w_i(\mathbb{E})\neq w_i(\mathbb{F})$, then $J(\mathbb{E})\neq J(\mathbb{F})$ for any bundles $\mathbb{E}$, $\mathbb{F}$ over $\Lambda$. Thus, denoting by
\[w(\mathbb{E})=1+w_1(\mathbb{E})+w_2(\mathbb{E})+\ldots\in H^\ast(\Lambda;\mathbb{Z}_2)\]
the total Stiefel-Whitney class of $\mathbb{E}$, we obtain the following corollary of Theorem \ref{thm:nonlin}.
\begin{cor}
If the system \eqref{nonlinear-equation} satisfies the assumptions \emph{(F0)--(F3)} and
\[w(\IM \mathbb{P}^+(\overline{\kappa}))\neq w(\IM \mathbb{P}^-(\underline{\kappa}))\in H^\ast(\Lambda;\mathbb{Z}_2),\]
then there is a bifurcation point.
\end{cor}
\noindent
As in \cite{FiPejsachowiczII}, \cite{JacoboTMNAII} and \cite{NilsBif}, the non-vanishing of a Stiefel-Whitney class actually can show more than just the existence of a single bifurcation point.
\begin{theorem}\label{thm:nonlinII}
If $\Lambda$ is a compact connected topological manifold of dimension $k\geqslant 2$, the system \eqref{nonlinear-equation} satisfies the assumptions \emph{(F0)--(F3)}  and
\begin{equation*}
w_i(\IM \mathbb{P}^+(\overline{\kappa}))\neq w_i(\IM \mathbb{P}^-(\underline{\kappa}))\in H^i(\Lambda;\mathbb{Z}_2)
\end{equation*}
for some $1\leqslant i\leqslant k-1$, then the covering dimension of $\mathcal{B}$ is at least $k-i$ and $\mathcal{B}$ is not contractible as a topological space.
\end{theorem}
\begin{proof}
This follows from Theorem \ref{thm:nonlin} as in \cite[Theorem 2.7]{SkibaIch}.
\end{proof}
\noindent
Let us note that if we assume instead of $(F2)$ the more restrictive assumption that the linear vector field $\mathbb{A}$ is asymptotically hyperbolic, then we obtain from Theorem \ref{thm:nonlin}, Theorem \ref{thm:nonlinII} and Corollary \ref{cor:asympFred} the bifurcation theorems of our previous work \cite{SkibaIch}. We provided in \cite{SkibaIch} an example showing that the bifurcation theorems obtained in that reference are not true without assuming that the linearized operators are invertible for one parameter value $\lambda_0\in\Lambda$.  Consequently, by Remark \ref{rem-iso},  $(F3)$ cannot be lifted in Theorem \ref{thm:nonlin} and Theorem \ref{thm:nonlinII}.\\
The following results prepare the examples at the end of this section.

\begin{lemma}\label{construction}
For any $k$-dimensional subbundle $\mathbb{E}$ of the trivial bundle $\Theta(\R^d)$ over a compact space $\Lambda$, $0<k<d$, and $q\in (0,1)$ there exists a continuous family $\mathbb{H}_{\mathbb{E}}\colon \Lambda\to GL(d,\R)\cap H(d,\R)$ such that
\begin{align*}
&E^s_q(\mathbb{H}_{\mathbb{E}}):=\{(\lambda,v)\in \Theta(\R^d)\mid \mathbb{H}_{\mathbb{E}}(\lambda)^n v= q^nv\text{ for }n\geqslant 0\}=\mathbb{\mathbb{E}},\\
&E^u_q(\mathbb{H}_{\mathbb{E}}):=\{(\lambda,v)\in \Theta(\R^d)\mid \mathbb{H}_{\mathbb{E}}(\lambda)^{n} v=q^{-n}v \text{ for }n\geqslant 0\}=\mathbb{E}^{\perp},
\end{align*}
where $\mathbb{E}^{\perp}:=\{(\lambda,w)\in \Lambda\times \R^d\mid w\perp v \text{ for all } v\in\mathbb{E}_{\lambda}\subset\R^d\}$.
\end{lemma}
\begin{proof}
The fiberwise orthogonal complement $\mathbb{E}^{\perp}$ is a $(d-k)$-dimensional vector subbundle of $\Theta(\R^d)$ such that
$\mathbb{E}\oplus \mathbb{E}^{\perp}\approx\Theta(\R^d)$ (cf. \cite[p.~354]{Dieck} or \cite{Hatcher1}). Consequently, every vector $v\in \R^d$ can be represented in the unique form
\begin{equation}\label{vector}
v=v_{\lambda}+v_{\lambda}^{\perp}, \text{where } v_{\lambda}\in \mathbb{E}_{\lambda}\subset\R^d \text{ and } v_{\lambda}^{\perp}\in \mathbb{E}_{\lambda}^{\perp}\subset \R^d.
\end{equation}
Moreover, the two maps $\Theta(\R^d)\ni (\lambda,v)\mapsto v_{\lambda}\in \mathbb{E}_{\lambda}\subset\R^d$ and $\Theta(\R^d)\ni (\lambda,v)\mapsto v_{\lambda}^{\perp}\in \mathbb{E}^{\perp}\subset\R^d$ are continuous, and hence uniformly continuous on the set $\{(\lambda,v)\in\Theta(\R^d)\mid |v|=1\}$.\\
We define $\mathbb{H}_{\mathbb{E}}\colon \Lambda\to GL(d,\R)\cap H(d,\R)$ by
\begin{align}\label{L-E}
\mathbb{H}_{\mathbb{E}}(\lambda)v:=q\cdot v_{\lambda}+(1/q)\cdot v_{\lambda}^{\perp},
\end{align}
for all $\lambda\in\Lambda$ and $v\in \R^d$, where $q\in (0,1)$ is a fixed constant. Now

\begin{align*}
||\mathbb{H}_{\mathbb{E}}(\lambda_1)-\mathbb{H}_{\mathbb{E}}(\lambda_2)||&=\sup_{|v|=1}|\mathbb{H}_{\mathbb{E}}(\lambda_1)v-\mathbb{H}_{\mathbb{E}}(\lambda_2)v|=
\sup_{|v|=1}|q\cdot (v_{\lambda_1}-v_{\lambda_2})-(1/q)(v_{\lambda_1}^{\perp}-v_{\lambda_2}^{\perp})|\\
&\leqslant \sup_{|v|=1}|q\cdot (v_{\lambda_1}-v_{\lambda_2})|+\sup_{|v|=1} |(1/q)(v_{\lambda_1}^{\perp}-v_{\lambda_2}^{\perp})|,
\end{align*}
which implies that $\mathbb{H}_{\mathbb{E}}$ is continuous. Finally, it is easy to see that the defined map $\mathbb{H}_{\mathbb{E}}$ has the required properties. This completes the proof.
\end{proof}
\noindent
The Atiyah-J\"anich Theorem states that for every Hilbert space $H$ and compact topological space $\Lambda$, the index bundle induces a bijection from the homotopy classes $[\Lambda,\Phi(H)]$ to $KO(\Lambda)$. Using a classical theorem of Arlt from \cite{Arlt}, we have noted in \cite{SkibaIch} that the same assertion is true for $\ell_0(\mathbb{R}^d)$.

\begin{theorem}[\cite{SkibaIch}]\label{A-J-bijection}
The index bundle induces a bijection
\begin{equation}\label{indbij}
\ind\colon [\Lambda,\Phi(\ell_0(\mathbb{R}^d))]\rightarrow KO(\Lambda).
\end{equation}
\end{theorem}
\noindent
A new observation is that we can reprove the surjectivity of \eqref{indbij} by using the operators \eqref{linear-equation}. This will be important for the construction of non-trivial examples below.

\begin{theorem}\label{realization}
For any two finite-dimensional subbundles $\mathbb{E}$ and $\mathbb{F}$ of the trivial bundle $\Theta(\R^d)$ over a compact space $\Lambda$, there exists
a linear parametrized vector field $\mathbb{A}\colon \Lambda\times \Z\to GL(d,\R)$ such that $L\colon  \Lambda\to \mathcal{L}(\ell_0(\R^d))$ given by  $L(\lambda)=\mathbb{S}_l- \mathcal{N}_{\mathbb{A}}(\lambda)$ is a continuous family of Fredholm operators with
\begin{equation*}
\ind(L)=[\mathbb{E}]-[\mathbb{F}]\in KO(\Lambda).
\end{equation*}
\end{theorem}

\begin{proof}
At the beginning of the proof, let us observe that Lemma \ref{construction} implies the existence of two functions
$\mathbb{A}_{\mathbb{E}}\colon \Lambda\to GL(d,\R)$ and $\mathbb{A}_{\mathbb{F}}\colon \Lambda\to GL(d,\R)$ such that
\begin{align*}
E^s_q(\mathbb{A}_{\mathbb{E}})=\mathbb{E},\;E^u_q(\mathbb{A}_{\mathbb{E}})=\mathbb{E}^{\perp} \text{ and }
E^s_q(\mathbb{A}_{\mathbb{F}})=\mathbb{F},\;E^u_q(\mathbb{A}_{\mathbb{F}})=\mathbb{F}^{\perp},
\end{align*}
where $q\in (0,1)$ is fixed. Moreover, arguing as in the proof of Lemma \ref{construction}, we find for $\kappa_+>0$ and $\kappa_-<0$ two continuous projectors
\begin{align*}
\mathbb{P}^+\colon \Lambda\times \Z_{\kappa_+}^+\to \mathcal{L}(\R^d) \text{ and }
\mathbb{P}^-\colon \Lambda\times \Z_{\kappa_-}^-\to \mathcal{L}(\R^d)
\end{align*}
by
\begin{align*}
\mathbb{P}^+(\lambda,n)v=v_{\lambda}\in \mathbb{E}_{\lambda} \text{ and } \mathbb{P}^-(\lambda,n)w=w_{\lambda}\in \mathbb{F}_{\lambda},
\end{align*}
where $v=v_{\lambda}+v_{\lambda}^{\perp}\in \mathbb{E}_{\lambda}\oplus \mathbb{E}_{\lambda}^{\perp}$ and $w=w_{\lambda}+w_{\lambda}^{\perp}\in \mathbb{F}_{\lambda}\oplus \mathbb{F}_{\lambda}^{\perp}$. Note that $(I_d-\mathbb{P}^+(\lambda,n))v=v_{\lambda}^{\perp}$ and $(I_d-\mathbb{P}^-(\lambda,n))w=w_{\lambda}^{\perp}$.
Let now $\mathbb{A}\colon \Lambda\times\Z\to GL(\R^d)$  be defined by
\begin{align*}
\mathbb{A}(\lambda,n):= \mathbbm{1}_{(-\infty,\kappa_-)}(n) \mathbb{H}_{\mathbb{F}}(\lambda)+
\mathbbm{1}_{[\kappa_-,\kappa_+]}(n)T(\lambda,n)+
\mathbbm{1}_{(\kappa_+,\infty)}(n)\mathbb{H}_{\mathbb{E}}(\lambda),
\end{align*}
where $\mathbb{H}_{\mathbb{E}}$ and $\mathbb{H}_{\mathbb{F}}$ are given as in \eqref{L-E} and $T\colon \Lambda\times (\Z\cap [\kappa_-,\kappa_+])\to GL(\R^d)$ is a finite family $\{T_n\colon \Lambda\to GL(\R^d)\mid \kappa_-\leqslant n\leqslant \kappa_+\}$ of continuous functions.
We will show that $\mathbb{A}$ with the two projectors $\mathbb{P}^-$ and $\mathbb{P}^+$  admits an ED both $\Z_{\kappa-}^-$ and $\Z_{\kappa_+}^+$.
For $\lambda\in\Lambda$, let $k\geqslant n$ with $k,n\in \Z_{\kappa_+}^+$. Then
\begin{align*}
&|\Phi_{\lambda}(k,n)\mathbb{P}^+(\lambda,n)v|=|\mathbb{A}_{\mathbb{E}}(\lambda)^{k-n}v_{\lambda}|=|q^{k-n}v_{\lambda}|=q^{k-n}|\mathbb{P}^+(\lambda,n)v|,\\
&|\Phi_{\lambda}(k,n)(I_d-\mathbb{P}^+(\lambda,n))v|=|\mathbb{A}_{\mathbb{E}}(\lambda)^{k-n}v_{\lambda}^{\perp}|=|(1/q)^{k-n}v_{\lambda}^{\perp}|=
(1/q)^{k-n}|(I_d-\mathbb{P}^+(\lambda,n))v|,
\end{align*}
which proves that $\mathbb{A}$ admits an ED on $\Z^+_{\kappa_+}$ with respect to $\mathbb{P}^+$. The same arguments can be repeated in order to show that
$\mathbb{A}$ admits an ED on $\Z^-_{\kappa_-}$ with respect to $\mathbb{P}^-$.
Consequently, Theorem \ref{ind-bundle-formula} implies
that $L\colon  \Lambda\to \mathcal{L}(\ell_{0}(\R^d))$ is a continuous family of Fredholm maps and
\begin{equation*}
\ind(L)=[\IM \mathbb{P}^{+}(\kappa_+)]-[\IM \mathbb{P}^{-}(\kappa_-)]=[\mathbb{E}]-[\mathbb{F}]\in KO(\Lambda).
\end{equation*}
This completes the proof.
\end{proof}
Let us now give the announced class of examples. Assume that $E$ and $F$ are subbundles of the trivial bundle $\Theta(\mathbb{R}^d)$ over the compact space $\Lambda$ having different total Stiefel-Whitney classes, i.e., $w(E)\neq w(F)\in H^\ast(\Lambda;\mathbb{Z}_2)$. We consider as in the proof of Theorem \ref{realization} the asymptotically hyperbolic family $\mathbb{A}\colon \Lambda\times \Z\to \mathcal{L}(\R^d)$ given by

\[\mathbb{A}(\lambda,n):=
\mathbbm{1}_{(-\infty,\kappa_-)}(n)\mathbb{H}_{F}(\lambda)+\mathbbm{1}_{[\kappa_-,\kappa_+]}(n)T(\lambda,n)+ \mathbbm{1}_{(\kappa_+,\infty)}(n)\mathbb{H}_{E}(\lambda),\]
where $\kappa_-<0<\kappa_+$ and $\{T_n\colon \Lambda\to GL(\R^d)\mid \kappa_-\leqslant n\leqslant \kappa_+\}$ is a finite family of continuous functions. We now take $\gamma_\pm$ as in our Theorem \ref{perturbed-ED} and assume that $\mathbb{D}\colon \Lambda\times \Z\to \mathcal{L}(\R^d)$ is $\gamma_\pm$ small for some $\widetilde{\kappa}_\pm\in\mathbb{Z}^\pm$.


Further, let $\mathbb{R}\colon \Lambda\times\Z\times \R^d\to \R^d$ satisfy $(F0)$, $(F1)$ and
\begin{itemize}
\item $\mathbb{R}(\lambda,n,0)=0$ for all $\lambda\in \Lambda$ and $n\in\Z$,
\item $(D_2 \mathbb{R}_n)(\lambda,0)\rightarrow 0$ as $n\rightarrow\pm\infty$ uniformly in $\lambda\in \Lambda$.
\end{itemize}
We now consider the discrete dynamical systems
\begin{align}\label{system2}
\phi(n+1)=(\mathbb{A}+\mathbb{D})(\lambda,n)\phi(n)+\mathbb{R}(\lambda,n,\phi(n)),\quad n\in\mathbb{Z}.
\end{align}
The linearization of \eqref{system2} at $0$ is $\mathbb{A}+\mathbb{D}$ and it follows from  Corollary \ref{ind-perturbation} and Theorem \ref{realization} that $L_{\mathbb{A}+\mathbb{D}}\colon \Lambda\to \mathcal{L}(\ell_0(\R^d))$ is a continuous family of Fredholm operators such that

\[\ind(L_{\mathbb{A}+\mathbb{D}})=\ind(L_{\mathbb{A}})=[E]-[F]\in \widetilde{KO}(\Lambda).\]
As $w(E)\neq w(F)$, we obtain from Theorem \ref{thm:nonlin} that if there exists $\lambda_0\in\Lambda$ such that $1\not\in \Sigma((\mathbb{A}+\mathbb{D})(\lambda_0))$ $($comp. Lemma \ref{iso}$)$, then there is a bifurcation point of \eqref{system2}. Moreover, if $\Lambda$ is a manifold of dimension $m$ and $w_i(E)\neq w_i(F)$ for some $1\leq i\leq m-1$, then we get from Theorem \ref{thm:nonlinII} that the covering dimension of the set $\mathcal{B}$ of all bifurcation points of \eqref{system2} is at least $m-i$ and $\mathcal{B}$ is not contractible as a topological space.  Finally, let us emphasize that we could not have got these results from the bifurcation theory developed in \cite{SkibaIch} as the family $\mathbb{A}+\mathbb{D}$ is in general not asymptotically hyperbolic.\\
Let us recall that $w_1(E)= 0\in H^1(\Lambda;\mathbb{Z}_2)$ if and only if $E$ is orientable. For example, if $\Lambda=S^1$, $E\subset\Theta(\mathbb{R}^2)$ is the M\"obius bundle and $F=S^1\times\mathbb{R}\subset\Theta(\mathbb{R}^2)$, then \eqref{system2} has a bifurcation point. An often used setting in nonlinear analysis is that $\Lambda=G_n(\mathbb{R}^{2n})$ is the Grassmannian of all $n$-dimensional subspaces of $\mathbb{R}^{2n}$ and $\gamma_n(\mathbb{R}^{2n})$ is the tautological bundle (cf. e.g., \cite{FiPejsachowiczII}, \cite{bartsch:91}, \cite{NilsBif}). As $\gamma_n(\mathbb{R}^{2n})$ is not orientable, we obtain from Theorem \ref{thm:nonlinII} that the set of bifurcation points of \eqref{system2} is non-contractible and at least of dimension $n^2-1$. Similarly,  if $\Lambda$ is a smooth non-orientable $m$-manifold, $m\geq 2$, $E=T\Lambda\subset\Theta(\mathbb{R}^d)$ and $F=\Theta(\mathbb{R}^m)$, then the set of bifurcation points of \eqref{system2} is non-contractible and at least of dimension $m-1$.\\
Finally, let us point out that there are various results ensuring the non-triviality of Stiefel-Whitney classes $w_k(T\Lambda)$, $k\geq 2$, for orientable manifolds $\Lambda$ of dimension $m$. For example, if $\Lambda$ has an odd Euler characteristic, then $w_m(T\Lambda)\neq 0$ (c.f. \cite[Cor. 11.12]{MiSta}) and we again obtain the existence of a bifurcation point of \eqref{system2} for $E=T\Lambda$ and $F=\Theta(\mathbb{R}^m)$. Note that Theorem \ref{thm:nonlinII} cannot be applied in this case as we would need that $w_k(T\Lambda)\neq 0$ for some $1\leq k\leq m-1$, however the recent paper \cite{Renee} indicates that there might very well be non-trivial lower Stiefel-Whitney classes in this case. Hence Theorem \ref{thm:nonlinII} can yield better results than just the existence of a single bifurcation point. For example, if $\Lambda$ has a CW-structure of the type

\begin{equation}\label{CW-K}
K=e^0\cup e^{r}\cup e^{2r}\cup\cdots\cup e^{nr},
\end{equation}
for $r=1,2$ or $4$ and some $n\geq 1$, then $H^{\ast}(\Lambda;\Z_2)\approx \Z_2[\alpha]/(\alpha^{n+1})$, where $\alpha$ is the non-zero class in $H^r(\Lambda;\mathbb{Z}_2)$, and the total Stiefel-Whitney class is given by $w(T\Lambda)=(1+\alpha)^{n+1}$. This applies in particular to real, complex and quaternionic projective spaces. Note that apart from even dimensional real projective spaces, all these manifolds are orientable and hence have vanishing first Stiefel-Whitney classes.

\subsection*{Acknowledgements}
The first author was partially supported by the Polish National Science Center (NCN), research grant no. DEC-
2017/01/X/ST1/00889. The project was co-financed by the Polish National Agency for Academic Exchange (NAWA) and
the German Academic Exchange Service (DAAD).

\vspace{0.5cm}
Robert Skiba\\
Faculty of Mathematics and Computer Science\\
Nicolaus Copernicus University in Toru\'n\\
Poland\\
E-mail: robert.skiba@mat.umk.pl

\vspace{0.5cm}

Nils Waterstraat\\
Martin-Luther-Universit\"at Halle-Wittenberg\\
Naturwissenschaftliche Fakult\"at II\\
Institut f\"ur Mathematik\\
06099 Halle (Saale)\\
Germany\\
E-mail: nils.waterstraat@mathematik.uni-halle.de

\end{document}